\documentclass[12pt,british,a4paper,reqno]{amsart}

	\usepackage[T1]{fontenc}
	\usepackage[latin9]{inputenc}
	\usepackage[british]{babel}

	\usepackage{amsmath}
	\usepackage{amssymb}
	\usepackage{amsthm}
	\usepackage{bbm}
	\usepackage{braket}
	\usepackage{xcolor}
	\usepackage[shortlabels]{enumitem}
	\usepackage{fancyhdr}
	\usepackage{graphicx}
	\usepackage{ifsym}
	\usepackage{indentfirst}
	\usepackage{lmodern}
	\usepackage{mathrsfs}
	\usepackage{mathtools}
	\usepackage{parskip}
	\usepackage{proof}
	\usepackage{qtree}
	\usepackage{scalerel}
	\usepackage[%
	]{setspace}
	\usepackage{tensor}
	\usepackage{tikz}
	\usepackage{tikz-cd}
	
	\usepackage{wasysym}
	\usepackage{oplotsymbl} 
	
	\usepackage[hyperfootnotes=false, hidelinks]{hyperref}
		\usepackage{bookmark}
	\usepackage{geometry}

	\bookmarksetup{numbered,open}

	\renewcommand{\geq}{\geqslant}
	\renewcommand{\leq}{\leqslant}
	\renewcommand{\phi}{\varphi}
	\renewcommand{\Im}{\operatorname{Im}\nolimits}
	\renewcommand{\sp}{\mathrm{sp}}

	\providecommand{\corollaryname}{Corollary}
	\providecommand{\definitionname}{Definition}
	\providecommand{\examplename}{Example}
	\providecommand{\lemmaname}{Lemma}
	\providecommand{\notationname}{Notation}
	\providecommand{\propositionname}{Proposition}
	\providecommand{\remarkname}{Remark}
	\providecommand{\theoremname}{Theorem}
	\providecommand{\setupname}{Setup}
	\providecommand{\conjecturename}{Conjecture}
	\providecommand{\questionname}{Question}
	\providecommand{\claimname}{Claim}

	\theoremstyle{plain}
		\newtheorem{thm}{\protect\theoremname}[section] 
		\newtheorem{thmx}{Theorem}
		\newtheorem{prop}[thm]{\protect\propositionname}
		\newtheorem{lem}[thm]{\protect\lemmaname}
		\newtheorem{cor}[thm]{\protect\corollaryname}

	\theoremstyle{definition}
		\newtheorem{defn}[thm]{\protect\definitionname}
		\newtheorem{notation}[thm]{\protect\notationname}
		\newtheorem{example}[thm]{\protect\examplename}
		\newtheorem{setup}[thm]{\setupname}

	\theoremstyle{remark}
		\newtheorem{rem}[thm]{\protect\remarkname}
		
	\numberwithin{figure}{section}
	\numberwithin{equation}{section}

	\makeatletter
	
		\newcommand\ackname{Acknowledgements}
		\newenvironment{acknowledgements}{%
			\medskip
			\bgroup
			\list{}{\labelwidth\z@
			\leftmargin3pc \rightmargin\leftmargin
			\listparindent\normalparindent \itemindent\z@
			\parsep\z@ \@plus\p@
			
				}%
				\Small
				\item[\hskip\labelsep\scshape\ackname.]%
			}{%
			\endlist\egroup
			}
			
	\makeatother

	\usetikzlibrary{matrix,arrows,decorations.pathmorphing,positioning,decorations.pathreplacing}
	\tikzset{commutative diagrams/.cd, 
		mysymbol/.style = {start anchor=center, end anchor = center, draw = none}}
	\tikzcdset{every label/.append style = {font = \footnotesize}}

	\newcommand{\commutes}[2][\circlearrowleft]{\arrow[mysymbol]{#2}[description]{#1}}

	\newcommand{\BE}{\mathbb{E}}
	\newcommand{\BF}{\mathbb{F}}

	\newcommand{\BZ}{\mathbb{Z}}

	\newcommand{\CA}{\mathcal{A}}
	\newcommand{\CB}{\mathcal{B}}
	\newcommand{\CC}{\mathcal{C}}
	\newcommand{\DD}{\mathcal{D}}
	
	\newcommand{\CF}{\mathcal{F}}
	\newcommand{\CG}{\mathcal{G}}
	
	\newcommand{\CI}{\mathcal{I}}

	\newcommand{\CS}{\mathcal{S}}
	\newcommand{\CT}{\mathcal{T}}

	\newcommand{\CX}{\mathcal{X}}

	\newcommand{\SF}{\mathscr{F}}

		\newcommand{\SET}{\operatorname{\mathsf{Set}}\nolimits}
		\newcommand{\Ab}{\operatorname{\mathsf{Ab}}\nolimits}
		\newcommand{\add}{\operatorname{\mathsf{add}}\nolimits}

		\newcommand{\ind}{\operatorname{\mathsf{ind}}\nolimits}

		\newcommand{\skel}{\mathrm{skel}}
		\newcommand{\op}{\mathrm{op}}
		
		\newcommand{\field}{k}
		
		\newcommand{\sse}{\subseteq}

		\newcommand{\obj}{\operatorname{obj}\nolimits}

		\newcommand{\Ker}{\operatorname{Ker}\nolimits}

		\newcommand{\iso}{\cong}
		\newcommand{\niso}{\ncong}

		\newcommand{\Hom}{\operatorname{Hom}\nolimits}

		\newcommand{\rad}{\operatorname{rad}\nolimits}

		\newcommand{\onto}{\rightarrow\mathrel{\mkern-14mu}\rightarrow}

		\newcommand{\idfunc}[1]{\mathbbm{1}_{#1}}

		\newcommand{\indx}[1]{\operatorname{\mathrm{ind}}_{#1}\nolimits}
		
		\newcommand{\indpent}{\mathchoice
  			{\pentago}
  			{\pentago}
  			{\scalebox{.7}{\pentago}}
  			{\scalebox{.5}{\pentago}}
		}
		\newcommand{\dindx}[1]{\operatorname{\mathrm{ind}}_{#1}^{\indpent}\nolimits}

		\newcommand{\Ext}{\operatorname{Ext}\nolimits}
		
		\newcommand{\fs}{\mathfrak{s}}
		
		\newcommand{\sus}{\Sigma} 

		\newcommand{\rmod}[1]{\operatorname{\mathsf{mod}}\nolimits{#1}}
		
		\newcommand{\rMod}[1]{\operatorname{\mathsf{Mod}}\nolimits{#1}}

		\newcommand{\fl}[1]{\operatorname{\mathsf{fl}}\nolimits{#1}}
		
		\newcommand{\kom}{\mathsf{K}}
		\newcommand{\com}{\mathsf{C}}

		\newcommand{\cone}{\mathrm{MC}}

		\newcommand{\combul}{\raisebox{0.5pt}{\scalebox{0.6}{$\bullet$}}}

	\newcommand{\kdual}{\operatorname{D}\nolimits}

	\newcommand{\deff}{\coloneqq}
	\newcommand{\eps}{\varepsilon}
	\newcommand{\lan}{\langle}
	\newcommand{\ran}{\rangle}

	\newcommand{\wt}[1]{\widetilde{#1}}
	\newcommand{\wh}[1]{\widehat{#1}}
	\newcommand{\ol}[1]{\overline{#1}}

	\newcommand\restr[2]{{\left.\kern-\nulldelimiterspace#1
						\right|_{#2}}}

	\makeatletter
	    \@namedef{subjclassname@2020}{\textup{2020} Mathematics Subject Classification}
	\makeatother

\begin{document}

\title[Grothendieck groups and a modified Caldero-Chapoton map]%
{
Grothendieck groups of $\lowercase{d}$-exangulated categories 
and a modified Caldero-Chapoton map}

\author[J{\o{}}rgensen]{Peter J{\o{}}rgensen}
	\address{
		Department of Mathematics\\
		Aarhus University\\
		Ny Munkegade 118\\
		8000 Aarhus C\\
		Denmark
	}
    \email{peter.jorgensen@math.au.dk}

\author[Shah]{Amit Shah}
    \email{amit.shah@math.au.dk}

\date{\today}

\keywords{%
	Caldero-Chapoton map, 
	Grothendieck group, 
	higher angulated category, 
	index, 
	$n$-cluster tilting subcategory, 
	$n$-exangulated category, 
	rigid subcategory%
}

\subjclass[2020]{%
Primary 16E20; 
Secondary 13F60, 18E05, 18E10, 18G80%
}

\begin{abstract}
A strong connection between cluster algebras and representation theory was established by the cluster category. Cluster characters, like the original Caldero-Chapoton map, are maps from certain triangulated categories to cluster algebras and they have generated much interest.
Holm and J\o{}rgensen constructed a modified Caldero-Chapoton map from a sufficiently nice triangulated category to a commutative ring, which is a generalised frieze under some conditions. 
In their construction, a quotient $K_{0}^{\sp}(\CT)/M$ of a Grothendieck group of a cluster tilting subcategory $\CT$ is used. 
In this article, we show that this quotient is the Grothendieck group of a certain extriangulated category, thereby exposing the significance of it and the relevance of extriangulated structures.
We use this to define another modified Caldero-Chapoton map that recovers the one of Holm--J\o{}rgensen. 

We prove our results in a higher homological context. Suppose $\CS$ is a $(d+2)$-angulated category with subcategories $\CX\sse\CT\sse\CS$, where $\CX$ is functorially finite and $\CT$ is $2d$-cluster tilting, satisfying some mild conditions. 
We show there is an isomorphism between 
the Grothendieck group $K_{0}(\CS,\BE_{\CX},\fs_{\CX})$ 
of the category $\CS$, equipped with the $d$-exangulated structure induced by $\CX$, 
and the quotient $K_{0}^{\sp}(\CT)/N$, where $N$ is the higher analogue of $M$ above. 
When $\CX=\CT$ the isomorphism is induced by the higher index with respect to $\CT$ introduced recently by J\o{}rgensen. Thus, in the general case, we can understand the map 
taking an object in $\CS$ to its $K_{0}$-class in $K_{0}(\CS,\BE_{\CX},\fs_{\CX})$ as a higher index with respect to the rigid subcategory $\CX$.
\end{abstract}

\maketitle


\section{Introduction}

Cluster algebras were introduced by Fomin and Zelevinsky in \cite{FominZelevinsky-cluster-algebras-I} and have since seen links to several different fields, such as integrable systems, Poisson geometry and particle physics. The cluster category associated to a hereditary algebra, defined in \cite{BMRRT-cluster-combinatorics}, is a categorification of the corresponding cluster algebra. This relationship is exhibited by the so-called \emph{Caldero-Chapoton map}, which is a cluster character that associates elements of the cluster algebra to certain types of objects of the cluster category; 
see, for example, 
\cite{CalderoChapoton-cluster-algebras-as-hall-algebras-of-quiver-representations}, 
\cite{CalderoKeller-from-triangulated-categories-to-cluster-algebras-II}, 
\cite{CalderoKeller-from-triangulated-categories-to-cluster-algebras-I}, 
\cite{Demonet-categorification-of-skew-symmetrizable-cluster-algebras}, 
\cite{DerksenWeymanZelevinsky-quivers-with-potentials-and-their-representations-I}, 
\cite{DerksenWeymanZelevinsky-quivers-with-potentials-and-their-representations-II}, 
\cite{FuKeller-on-cluster-algebras-with-coefficients-and-2-calabi-Yau-categories}, 
\cite{HolmJorgensen-generalized-friezes-and-a-modified-caldero-chapoton-map-depending-on-a-rigid-object-1}, 
\cite{HolmJorgensen-generalized-friezes-and-a-modified-caldero-chapoton-map-depending-on-a-rigid-object-2}, 
\cite{JorgensenPalu-a-caldero-chapoton-map-for-infinite-clusters}, 
\cite{Palu-cluster-characters-for-2-calabi-yau-triangulated-categories}, 
\cite{Plamondon-cluster-characters-for-cluster-categories-with-infinite-dimensional-morphism-spaces}, 
\cite{ZhouZhu-cluster-algebras-arising-from-cluster-tubes}. 

Let us recall a version of the Caldero-Chapoton map equivalent to the original of \cite[Sec.\ 3]{CalderoChapoton-cluster-algebras-as-hall-algebras-of-quiver-representations}. 
For unexplained notation, see Subsection~\ref{sec:conventions-notation} below. 
Let $\field$ be an algebraically closed field and let $\CC$ be a $\field$-linear, $\Hom$-finite, Krull-Schmidt, triangulated category that is 2-Calabi-Yau. 
Denote by $\sus$ the suspension functor of $\CC$. 
The Caldero-Chapoton map depends on the choice of a cluster tilting object $T\in\CC$, which we assume to be basic. 
Set $\CT = \add T$. 
Define the functor $G\colon \CC\to \rmod{\CT}$ by $G_{T}(C) = \CC(T,\sus C)$. 
Since $\CT$ is cluster tilting, for each object $C\in\CC$ there is a triangle 
$T^{0}\to T^{1}\to C \to \sus T^{0}$, and we put $\indx{\CT}(C) = [T^{1}] - [T^{0}]$, which is an element of $K_{0}^{\sp}(\CT)$ known as the \emph{index} (with respect to $\CT$) of $C$. 
There is a group homomorphism $\ol{\phi}\colon K_{0}(\rmod{\CT}) \to K_{0}^{\sp}(\CT)$, measuring how far the index is from being additive over triangles in $\CC$ (see \cite[Thm.\ 4.4]{Jorgensen-tropical-friezes-and-the-index-in-higher-homological-algebra}, or Section~\ref{sec:re-modified-CC-map}). 
Put $A = \BZ[\tensor[]{x}{_{T'}}, \tensor[]{x}{_{T'}^{-1}}\tensor[]{]}{_{T'\in\ind{\CT}}}$. 
Define maps $\eps\colon K_{0}^{\sp}(\CT) \to A$, $\alpha\colon \obj(\CC)\to A$ and $\beta\colon K_{0}(\rmod{\CT})\to A$ as follows: 
$\eps([T']) = \tensor[]{x}{_{T'}}$ for $T'\in\ind{\CT}$; 
$\alpha = \eps\circ \indx{\CT}$; and 
$\beta = \eps\circ \ol{\phi}$. 
Assume also that $\eps$ is ``exponential'' in the sense that $\eps(0) = 1$ and $\eps(e+f) = \eps(e)\eps(f)$. 
Then the formula of the Caldero-Chapoton map given in \cite[1.8]{JorgensenPalu-a-caldero-chapoton-map-for-infinite-clusters} is
\begin{equation}
\label{eqn:modified-CC-map}
\rho(C) 
	= \alpha(C) \sum_{e}\chi\left(\text{Gr}_{e}G_{T}(C)\right) \beta(e),
\end{equation}
where $\text{Gr}_{e}(G_{T}(C))$ is the Grassmannian of submodules of 
$\CC(T,\sus C)$ in $\rmod{\CT}$ with $K_{0}$-class $e$ in $K_{0}(\rmod{\CT})$, and $\chi$ is the Euler characteristic defined by \'{e}tale cohomology with proper support. 

From its initial context, the Caldero-Chapoton map has been studied and generalised. 
For example, in \cite{HolmJorgensen-generalized-friezes-and-a-modified-caldero-chapoton-map-depending-on-a-rigid-object-2} an extension of the map was given dealing with the rigid subcategory case and where $A$ is a general commutative ring. 
Let $\CX = \add R$ be a \emph{rigid} subcategory (i.e.\ $\CC(\CX,\sus \CX) = 0$) contained in the cluster tilting subcategory $\CT$. 
In this generalisation, Holm and J\o{}rgensen make the following substitutions; see Section~\ref{sec:re-modified-CC-map} for more details. 
\begin{enumerate}[(\roman*)]
	\item $\eps$ is replaced with a map 
$\ol{\eps}\colon K_{0}^{\sp}(\CT)/N \to A$, where
\begin{center}
	$N 
	= 
	\Braket{
	[X] - [Y] |   
	\begin{array}{l}
	\begin{tikzcd}[column sep=0.3cm,ampersand replacement=\&]
	T \arrow{r}\& Y \arrow{r} \& T^{*}\arrow{r} \& \sus T,
	\end{tikzcd}
	\begin{tikzcd}[column sep=0.3cm,ampersand replacement=\&]
	T^{*}\arrow{r} \& X \arrow{r} \& T\arrow{r} \& \sus T^{*}
	\end{tikzcd}\\
	\text{are exchange triangles with } 
	T\in\ind\CT\setminus\ind\CX 
	\end{array}
	}$.
\end{center}

\item $\alpha$ is replaced by $\ol{\eps} Q \indx{\CT}$, where $Q\colon K_{0}^{\sp}(\CT)\to K_{0}^{\sp}(\CT)/N$ is the canonical surjection.

\item $\beta$ is replaced by $\ol{\eps} \nu$, where $\nu\colon K_{0}(\rmod{\CX})\to K_{0}^{\sp}(\CT)/N$ is the unique homomorphism making 
\begin{center}
\vspace{-4pt plus 1pt minus 1pt}
$
\begin{tikzcd}
K_{0}(\rmod{\CT}) \arrow{r}{\ol{\phi}}\arrow{d}[swap]{\kappa}& K_{0}^{\sp}(\CT) \arrow{d}{Q}\\
K_{0}(\rmod{\CX}) \arrow[dotted]{r}{\nu}& K_{0}^{\sp}(\CT)/N
\end{tikzcd}
$
\end{center}
commute, in which $\kappa$ is induced by the inclusion $\CX \sse \CT$. 

\item $G_{T}$ and $\text{Gr}_{e}(G_{T}(C))$ are replaced by similarly defined analogues $G_{R}$ and $\text{Gr}_{e}(G_{R}(C))$, respectively. 
\end{enumerate}
With these adjustments, \eqref{eqn:modified-CC-map} is then called the \emph{modified Caldero-Chapoton map} in \cite{HolmJorgensen-generalized-friezes-and-a-modified-caldero-chapoton-map-depending-on-a-rigid-object-2}. 
Note that if $\CX = \CT$ is cluster tilting then $N=0$ and the domains of $\eps$ and $\ol{\eps}$ agree. 
Furthermore, \cite[Thm.\ A]{HolmJorgensen-generalized-friezes-and-a-modified-caldero-chapoton-map-depending-on-a-rigid-object-2} showed that the modified Caldero-Chapoton map $\rho$ is a \emph{generalised frieze} (see \cite[Def.\ 3.4]{HolmJorgensen-generalized-friezes-and-a-modified-caldero-chapoton-map-depending-on-a-rigid-object-1}) under suitable assumptions, thereby demonstrating this theory has particular implications in combinatorics. This map also recovers the combinatorial generalised friezes of \cite{BessenrodtHolmJorgensen-generalized-frieze-pattern-determinants-and-higher-angulations-of-polygons}; see also \cite{CanakciJorgensen-driezes-weak-friezes-and-T-paths}. 

Although $K_{0}^{\sp}(\CT)/N$ is observably a correct group to consider as the domain for $\ol{\eps}$, there has also been other considerable interest in this quotient. For instance, Palu proved that $K_{0}^{\sp}(\CT)/N$ recovers the Grothendieck group of the triangulated category $\CC$ when $\CX = 0$ (see \cite[Thm.\ 10]{Palu-grothendieck-group-and-generalized-mutation-rule-for-2-calabi-yau-triangulated-categories}). In this case, $N$ is generated by relations coming from all pairs of exchange triangles for indecomposable objects in $\CT$. These results were generalised to the higher angulated case by Fedele in \cite{Fedele-grothendieck-groups-of-triangulated-categories-via-cluster-tilting-subcategories}. 
However, there is not yet an explanation of the significance of $K_{0}^{\sp}(\CT)/N$, even though it shows up in different contexts, indicating its importance and some deeper connection. 

Motivated by this, and inspired by an approach taken in \cite{PadrolPaluPilaudPlamondon-associahedra-for-finite-type-cluster-algebras-and-minimal-relations-between-g-vectors}, in the present article we remedy this by giving an interpretation of $K_{0}^{\sp}(\CT)/N$ as the Grothendieck group of a certain \emph{extriangulated} category dependent on $\CX$; see Theorem~\ref{thmx:A}. In particular, this shows that this group is indeed \emph{the} correct group to use in the construction of the modified Caldero-Chapoton map and why. 
Conversely, the interest in $K_{0}^{\sp}(\CT)/N$ and its appearance as the Grothendieck group of an extriangulated category illustrates the relevance of extriangulated and, more generally, higher exangulated structures. 
Although we prove our results in the more general setting of higher homological algebra, we will remain in the classical setting for the purposes of this introduction, which is also the setting of Section~\ref{sec:re-modified-CC-map}. 

Since $\CC$ is triangulated, we have that $(\CC,\BE,\fs)$ is an extriangulated category, where $\BE = \CC(-,\sus -)$ is a biadditive functor and $\fs$ is the (canonical) realisation of $\BE$ (see \cite[Prop.\ 3.22]{NakaokaPalu-extriangulated-categories-hovey-twin-cotorsion-pairs-and-model-structures} or Example~\ref{example:n+2-angulated-category-is-n-exangulated}). 
There is a subfunctor of $\BE$ given as follows. 
Define $\BE_{\CX}$ on objects by 
\[
\BE_{\CX}(C,A) 
	= \set{ \delta \in \CC(C,\sus A) 
			| \delta\circ\gamma = 0 \text{ for all } \gamma\colon X\to C \text{ with } X\in\CX}.
\]
Then, by equipping $\CC$ with $\BE_{\CX}$ instead, we obtain another extriangulated structure 
$(\BE_{\CX},\fs_{\CX})$ on $\CC$, where $\fs_{\CX}$ is the restriction of $\fs$ to $\BE_{\CX}$. 
The group $K_{0}(\CC,\BE_{\CX},\fs_{\CX})$ is then defined to be the quotient of $K_{0}^{\sp}(\CC)$ by the subgroup generated by elements $[A]-[B]+[C]$, where 
$\begin{tikzcd}[column sep = 0.5cm]
A\arrow{r} & B\arrow{r} & C\arrow{r}{\delta} & \sus A
\end{tikzcd}$ 
is a triangle in $\CC$ with $\delta\in\BE_{\CX}(C,A)$. 
See Section~\ref{sec:n-exangulated-categories} for more details. 
As a special case of Theorem~\ref{thm:G-R-is-an-isomorphism}, we have: 

\begin{thmx}
\label{thmx:A}

There is an isomorphism $K_{0}^{\sp}(\CT)/N \iso K_{0}(\CC,\BE_{\CX},\fs_{\CX})$. 

\end{thmx}

In case $\CX = \CT$, the isomorphism $K_{0}(\CC,\BE_{\CT},\fs_{\CT}) \to K_{0}^{\sp}(\CT)$ is induced by the index $\indx{\CT}\colon\obj(\CC) \to K_{0}^{\sp}(\CT)$; see \cite[Prop.\ 4.11]{PadrolPaluPilaudPlamondon-associahedra-for-finite-type-cluster-algebras-and-minimal-relations-between-g-vectors}. 
The authors are very grateful to P.-G. Plamondon, who pointed out this result from \cite{PadrolPaluPilaudPlamondon-associahedra-for-finite-type-cluster-algebras-and-minimal-relations-between-g-vectors} to the first author. This led to the development of Theorem~\ref{thm:d-index-gives-isomorphism-G-group-of-T-induced-exangulated-structure-on-S-to-split-G-group-of-T}, which is a higher homological version. 
This cluster tilting special case suggests that we can think of the surjection 
$Q_{\CX}\colon K_{0}^{\sp}(\CC) \to K_{0}(\CC,\BE_{\CX},\fs_{\CX})$ 
as an index with respect to $\CX$. 
Furthermore, there is also an additivity formula with error-term for the homomorphism $Q_{\CX}$ (see 
\cite{JorgensenShah-the-index-with-respect-to-a-rigid-subcategory}). 
Moreover, using $Q_{\CX}$ and a map that also behaves like an error-term map under suitable conditions (see Remark~\ref{rem:psi-is-theta-X-mu}), we give a new version of the Caldero-Chapoton map, which we call the \emph{$\CX$-Caldero-Chapoton map}; see Definition~\ref{def:X-CC-map}.
This recovers the modified Caldero-Chapoton map from \cite{HolmJorgensen-generalized-friezes-and-a-modified-caldero-chapoton-map-depending-on-a-rigid-object-2}.

This paper is organised as follows. 
In Section~\ref{sec:n-exangulated-categories}, we cover the definitions and results we will need on $n$-exangulated categories, including their Grothendieck groups and $n$-exangulated structures induced by subcategories. 
In Section~\ref{sec:cluster-tilting-subcategories} we show, for an Oppermann-Thomas cluster tilting subcategory $\CT$ of a $(d+2)$-angulated category $\CS$, there is an isomorphism between the split Grothendieck group of $\CT$ and the Grothendieck group of $\CS$ with $d$-exangulated structure induced by $\CT$. 
In Section~\ref{sec:K0TmodNX-as-a-grothendieck-group} we prove the higher version of Theorem~\ref{thmx:A}.
Lastly, in Section~\ref{sec:re-modified-CC-map} we apply our results to the classical setting and define the $\CX$-Caldero-Chapoton map.

\subsection{Conventions and notation}
\label{sec:conventions-notation}

\begin{enumerate}[label=(\roman*)]

	\item $\Ab$ denotes the category of all abelian groups. 
	
	\item By an \emph{additive subcategory}, we will mean a full subcategory that is closed under isomorphisms, direct sums and direct summands. Note that if $\CB$ is an additive subcategory of $\CA$ and $\CA$ is idempotent complete, then $\CB$ is also idempotent complete since it is closed under direct summands.
	
	\item Let $\CA$ be an additive category. For a subcategory $\CB$ of $\CA$ that is closed under finite direct sums, $[\CB]$ denotes the two-sided ideal of $\CA$ 
	consisting of morphisms factoring through $\CB$. 

	Fix a skeleton $\CA^{\skel}$ of $\CA$. We denote by $\ind \CA$ the class of objects of $\CA^{\skel}$ consisting of only indecomposable objects. 
	Furthermore, if $\CB\sse\CA$ is an additive subcategory, we assume that skeletons are chosen in a compatible way, i.e.\ $\CB^{\skel}\sse\CA^{\skel}$, so in particular $\ind\CB\sse\ind\CA$.

	\item For a field $\field$, we denote by $\rMod{\field}$ (respectively, $\rmod{\field}$) the category of all (respectively, finite-dimensional) $\field$-vector spaces. 
	
	\item Suppose $\CC$ is a skeletally small, $\field$-linear category. 
	We denote by $\rMod{\CC}$ the category of all $\field$-linear, contravariant functors $\CC\to \rMod{\field}$. 
	The category $\rMod{\CC}$ is $\field$-linear and abelian (see, for example, \cite[Thm.\ 10.1.3]{Prest-purity-spectra-localisation} or \cite[Sec.\ 3, Thm.\ 4.2]{Popescu-abelian-cats-with-apps-to-rings-and-modules}). 
	
	A functor $M\in\rMod{\CC}$ is called \emph{finitely presented} if there is an exact sequence 
	$\begin{tikzcd}[column sep=0.5cm]
	\CC(-,A) \arrow{r}& \CC(-,B) \arrow{r}& M \arrow{r}& 0
	\end{tikzcd}$
	for some $A,B\in\CC$; see \cite[p.\ 155]{Beligiannis-on-the-freyd-cats-of-additive-cats}. 
	(These functors have also been called `coherent' in e.g.\ \cite{Auslander-coherent-functors} and \cite{IyamaYoshino-mutation-in-tri-cats-rigid-CM-mods}. However, there is a more general notion of a coherent functor, see \cite[Def.\ B.5]{Fiorot-n-quasi-abelian-categories-vs-n-tilting-torsion-pairs}. This notion and what we call finitely presented agree when, for example, $\CC$ is also idempotent complete and has weak kernels. See \cite[App.\ B]{Fiorot-n-quasi-abelian-categories-vs-n-tilting-torsion-pairs} and the references therein for more details.) 
	We denote by $\rmod{\CC}$ the full subcategory of $\rMod{\CC}$ consisting of all the finitely presented 
	$\field$-linear, contravariant functors $\CC\to \rMod{\field}$. 
	The category $\rmod{\CC}$ is also $\field$-linear. 
	If $\CC$ is also idempotent complete and has weak kernels (e.g.\ $\CC$ is a contravariantly finite, additive subcategory of an idempotent complete, triangulated category, see \cite[Def.\ 2.9]{IyamaYoshino-mutation-in-tri-cats-rigid-CM-mods}), then $\rmod{\CC}$ is abelian and the inclusion functor $\rmod{\CC} \to \rMod{\CC}$ is exact. 
	See \cite[Sec.\ III.2]{Auslander-representation-dimension-of-artin-algebras-QMC}; see also \cite[Sec.\ 2]{Auslander-coherent-functors} and \cite{Auslander-selected-works-part-1}.

Finally, $\fl{\CC}$ denotes the full subcategory of $\rMod{\CC}$ consisting of objects of \emph{finite length} (see \cite[Sec.\ 5]{Krause-KS-cats-and-projective-covers}). 
	
	\item
	\label{1-1-vi}
	Let $\CS$ be a $\field$-linear category and suppose $\DD$ is a skeletally small, additive subcategory of $\CS$. 
	We let $F_{\DD}\colon \CS \to \rmod{\DD}$ denote the covariant functor defined on objects by sending $X\in\CS$ to $\restr{\CS(-, X)}{\DD}$. 
	Denote by $\DD^{\perp_{0}}\deff \Ker F_{\DD}$ the full subcategory of $\CS$ consisting of objects $X$ for which $\CS(\DD,X)=0$.
	
\end{enumerate}


\section{\texorpdfstring{$n$}{n}-exangulated categories}
\label{sec:n-exangulated-categories}
\subsection{The definition}
\label{sec:definition-of-n-exangulated-category}
Let $n\geq 1$ be a positive integer. 
Exact, abelian and triangulated categories (see \cite{Quillen-higher-algebraic-k-theory-I}, \cite{Grothendieck-tohoku-paper-abelian-categories} and \cite{Verdier-derived-cats-of-abelian-cats}, respectively) form the core of homological algebra. Indeed, in Sections~\ref{sec:cluster-tilting-subcategories} through \ref{sec:re-modified-CC-map} in this paper we are in a setting where there is some ambient triangulated category; see Setup~\ref{setup:initial-d-setup-C-T-S}. 
These ideas have been generalised to higher dimensions by Geiss--Keller--Oppermann in \cite{GeissKellerOppermann-n-angulated-categories}, where \emph{$(n+2)$-angulated} categories are introduced, and by Jasso in \cite{Jasso-n-abelian-and-n-exact-categories}, where \emph{$n$-abelian} and \emph{$n$-exact} categories are defined. 
In a parallel direction, the theories of exact and triangulated categories have been unified with the introduction of \emph{extriangulated} categories by Nakaoka--Palu in \cite{NakaokaPalu-extriangulated-categories-hovey-twin-cotorsion-pairs-and-model-structures}. 

Naturally, the counterpart in higher homological algebra of extriangulated categories has also been developed---namely, \emph{$n$-exangulated} categories as introduced by Herschend--Liu--Nakaoka in \cite{HerschendLiuNakaoka-n-exangulated-categories-I-definitions-and-fundamental-properties}. 
In this subsection, we briefly recall the definition of an $n$-exangulated category (see Definition~\ref{def:n-exangulated-category}). For more details, we refer the reader to \cite[Sec.\ 2]{HerschendLiuNakaoka-n-exangulated-categories-I-definitions-and-fundamental-properties}. 
We also recall how one can view an $(n+2)$-angulated category as an $n$-exangulated category (see Example~\ref{example:n+2-angulated-category-is-n-exangulated}).

Let $\CC$ be an additive category equipped with a biadditive functor $\BE\colon\CC^{\op}\times\CC\to \Ab$. 
A prototypical example when $n=1$ of an $n$-exangulated category (or an extriangulated category in the sense of \cite{NakaokaPalu-extriangulated-categories-hovey-twin-cotorsion-pairs-and-model-structures}) 
is an exact category. 
In this case, $\BE$ models the functor $\Ext^{1}$, which motivates the following terminology. 

\begin{defn}

\begin{enumerate}[(\roman*)]

	\item For any $A,C\in\CC$, any element of $\BE(C,A)$ is called an \emph{extension}.
	
	\item Suppose $\delta\in\BE(C,A)$ is an extension, and let $a\colon A\to A'$ and $c\colon C'\to C$ be arbitrary morphisms in $\CC$. 
	We define 
	$a_{*}\delta:=\BE(C,a)(\delta)\in\BE(C,A')$ and 
	$c^{*}\delta:=\BE(c,A)(\delta)\in\BE(C',A).$
	
	\item Suppose $\delta\in \BE(C,A)$ and $\eta\in\BE(D,B)$ are extensions. 
	A \emph{morphism of extensions} $\delta\to\eta$ is a pair $(a,c)$ of morphisms $a\colon A\to B$ and $c\colon C\to D$ in $\CC$, such that $ a_{*}\delta  = c^{*}\eta$.

\end{enumerate}
\end{defn}

Let $\delta\in\BE(C,A)$ be an extension. 
By the Yoneda Lemma, there is a natural transformation
$\delta^{\sharp}\colon \CC(A,-) \Rightarrow \BE(C,-)$, given by 
$(\delta^{\sharp})_{B}(a)=a_{*}\delta$ 
for each object $B\in\CC$ and each morphism $a\colon A\to B$. 
Dually, there is also a natural transformation 
$\delta_{\sharp}\colon \CC(-,C)\Rightarrow \BE(-,A)$.

We denote the \emph{category of (co)complexes} in $\CC$ by $\com_{\CC}$. 
Note that we use cohomological numbering throughout. 
Since we will be interested in equipping categories with certain classes of complexes with $n+2$ terms, it is helpful to introduce the following: 
the full subcategory of complexes concentrated in (cohomological) degrees $0,1,\ldots, n, n+1$ will be denoted by $\com_{\CC}^{n}$. 
Since a morphism $f^{\combul}\colon X^{\combul} \to Y^{\combul}$ in $\com_{\CC}^{n}$ will be of the form $(\ldots,0,0,f^{0},f^{1},\ldots,f^{n+1},0,0,\ldots)$, we will simply write $f^{\combul}=(f^{0},\ldots,f^{n+1})$. 
Moreover, we will restrict our attention to objects of $\com_{\CC}^{n}$ that also have a connection with an extension.

\begin{defn}
\label{def:n-exangles}

Let $X^{\combul}\in\com_{\CC}^{n}$ and $\delta\in\BE(X^{n+1},X^{0})$. The pair $\langle X^{\combul},\delta\rangle$ is an $n$-\emph{exangle} if
\[
\begin{tikzcd}[column sep=1.6cm]
\CC(-,X^{0})\arrow[Rightarrow]{r}[yshift=2pt]{\CC(-,d_{X}^{0})}&\CC(-,X^{1})\arrow[Rightarrow]{r}[yshift=2pt]{\CC(-,d_{X}^{1})}&\cdots\arrow[Rightarrow]{r}[yshift=2pt]{\CC(-,d_{X}^{n})}&\CC(-,X^{n+1})\arrow[Rightarrow]{r}[yshift=2pt]{\delta_{\sharp}}& \BE(-,X^{0})
\end{tikzcd}
\] 
and 
\[
\hspace*{-0.1cm}
\begin{tikzcd}[column sep=1.65cm]
\CC(X^{n+1},-)\arrow[Rightarrow]{r}[yshift=2pt]{\CC(d_{X}^{n},-)}&\CC(X^{n},-)\arrow[Rightarrow]{r}[yshift=2pt]{\CC(d_{X}^{n-1},-)}&\cdots\arrow[Rightarrow]{r}[yshift=2pt]{\CC(d_{X}^{0},-)}&\CC(X^{0},-)\arrow[Rightarrow]{r}[yshift=2pt]{\delta^{\sharp}}& \BE(X^{n+1},-)
\end{tikzcd}
\] 
are exact sequences of functors.

\end{defn}

\begin{defn}

Fix objects $A,C$ in $\CC$. The (not necessarily full) subcategory $\com_{(A,C)}^{n}$ of $\com_{\CC}^{n}$ is defined as follows: 
an object of $\com_{(A,C)}^{n}$ is a complex $X^{\combul}\in\com_{\CC}^{n}$ for which $X^{0}=A$ and $X^{n+1}=C$; and 
a morphism 
$f^{\combul}\colon X^{\combul}\to Y^{\combul}$ 
in $\com_{(A,C)}^{n}$ is a commutative diagram
\[
\begin{tikzcd}
A \arrow{r}\arrow[equals]{d}{}&X^{1}\arrow{d}{f^{1}} \arrow{r}& \cdots\arrow{r}& X^{n}\arrow{d}{f^{n}}\arrow{r} & C \arrow[equals]{d}{}\\
A \arrow{r}&Y^{1} \arrow{r}& \cdots\arrow{r}& Y^{n}\arrow{r} & C
\end{tikzcd}
\]
in $\CC$.

\end{defn}

For $A,C\in\CC$ and $X^{\combul},Y^{\combul}\in\com_{(A,C)}^{n}$, the usual notion of a \emph{homotopy} between morphisms in $\com_{\CC}$ gives an equivalence relation on $\com_{(A,C)}^{n}(X^{\combul},Y^{\combul})$, which we denote by $\sim$. 
Thus, we can form a specialised homotopy category, denoted $\kom^{n}_{(A,C)}$, which has the same objects as $\com_{(A,C)}^{n}$ and, for $X^{\combul},Y^{\combul}\in\kom^{n}_{(A,C)}$, we set 
$\kom^{n}_{(A,C)}(X^{\combul},Y^{\combul})\deff \com_{(A,C)}^{n}(X^{\combul},Y^{\combul})/{\sim}$. 
Furthermore, we call $f^{\combul}\in\com_{(A,C)}^{n}(X^{\combul},Y^{\combul})$ a \emph{homotopy equivalence} if its image in $\kom^{n}_{(A,C)}(X^{\combul},Y^{\combul})$ is an isomorphism. 
In this case, the complexes $X^{\combul}$ and $Y^{\combul}$ are said to be \emph{homotopy equivalent}, and the homotopy equivalence class of $X^{\combul}$ in $\com_{(A,C)}^{n}$ is denoted  $[X^{\combul}]$. This is \emph{not} necessarily the usual homotopy equivalence class of $X^{\combul}$ in $\com_{\CC}^{n}$; see \cite[Rem.\ 2.18]{HerschendLiuNakaoka-n-exangulated-categories-I-definitions-and-fundamental-properties}.

Just as $\Ext^{1}$ groups in an abelian category can be viewed as equivalence classes of short exact sequences, in an $n$-exangulated category homotopy equivalence classes will be associated to extensions.

\begin{defn}
\label{def:exact-realisation}

Let $\fs$ be a correspondence, which assigns to each $\delta\in\BE(C,A)$, for each $A,C\in\CC$, a homotopy equivalence class $\fs(\delta)=[X^{\combul}]$ for some  $X^{\combul}\in\com_{(A,C)}^{n}$. 
Then $\fs$ is called an \emph{exact realisation of $\BE$} if the following three conditions are met. 

\begin{enumerate}[label=(R\arabic*)]
\setcounter{enumi}{-1}

    \item\label{R0}
    For each $\delta\in\BE(C,A), \eta\in\BE(D,B)$ such that $\fs(\delta)=[X^{\combul}],\ \fs(\eta)=[Y^{\combul}]$, and for each morphism $(a,c)\colon \delta\to\eta$ of extensions, there exists a morphism $f^{\combul}=(f^{0},\ldots,f^{n+1})\in\com_{\CC}^{n}(X^{\combul},Y^{\combul})$ \emph{realising} $(a,c)$, i.e.\ with $f^{0}=a$ and $f^{n+1}=c$.

    \item\label{R1}
    If $\fs(\delta)=[X^{\combul}]$, then $\langle X^{\combul},\delta\rangle$ is an $n$-exangle.

    \item\label{R2}
    Let $A\in\CC$. Consider the group identity elements $_{A}0_{0}\in \BE(0,A)$ and $_{0}0_{A}\in\BE(A,0)$. Then 
    \begin{center}
    \vspace{-6pt plus 1pt minus 1pt}
    $\fs(_{A}0_{0})=[
    \begin{tikzcd}[column sep=0.7cm] A\arrow{r}{1_{A}}&A\arrow{r}&0\arrow{r}&\cdots\arrow{r}&0\end{tikzcd}
    ]$
    \end{center}
        \vspace{-4pt plus 1pt minus 1pt}
    and 
    \begin{center}
\vspace{-6pt plus 1pt minus 1pt}
    $\fs(_{0}0_{A})=[
    \begin{tikzcd}[column sep=0.7cm] 0\arrow{r}&\cdots\arrow{r}&0\arrow{r}&A\arrow{r}{1_{A}}&A\end{tikzcd}
    ].$
    \end{center}
    
\end{enumerate}

\end{defn}

If $\fs$ is an exact realisation of $\BE$ and 
\[
\fs(\delta) 
	= [X^{\combul}] 
	= [
	\begin{tikzcd}[column sep=0.6cm]
	X^{0} \arrow{r}{d_{X}^{0}}& X^{1} \arrow{r}& \cdots \arrow{r}& X^{n}\arrow{r}{d_{X}^{n}}& X^{n+1}
	\end{tikzcd}
	],
\]
then we call $d_{X}^{0}$ an \emph{$\fs$-inflation} and $d_{X}^{n}$ an \emph{$\fs$-deflation}. 

In order to simplify the statement of the main definition in this subsection, we recall the following.

\begin{defn}
\label{def:mapping-cone-as-in-Jasso}

\cite[Def.\ 2.11]{Jasso-n-abelian-and-n-exact-categories}
The \emph{mapping cone} $M^{\combul}\deff \cone(f)^{\combul}$ of a morphism  $f^{\combul}\in\com_{\CC}^{n-1}(X^{\combul},Y^{\combul})$ is the complex
\[
\begin{tikzcd}
X^{0}\arrow{r}{d_{M}^{-1}}& X^{1}\oplus Y^{0} \arrow{r}{d_{M}^{0}}&X^{2}\oplus Y^{1}\arrow{r}{d_{M}^{1}}& \cdots\arrow{r}{d_{M}^{n-2}}&X^{n}\oplus Y^{n-1} \arrow{r}{d_{M}^{n-1}}& Y^{n}
\end{tikzcd}
\]
in $\com_{\CC}^{n}$, where $d_{M}^{-1}\deff \begin{psmallmatrix}-d_{X}^{0}\\f^{0} \end{psmallmatrix}$,
\[
d_{M}^{i}\deff
\begin{pmatrix}-d_{X}^{i+1} & 0\\f^{i+1} & d_{Y}^{i} \end{pmatrix}
\]
for $i\in\{0,\ldots,n-2\}$, and $d_{M}^{n-1}\deff(\,f^{n}\;\;\, d_{Y}^{n-1})$.

\end{defn}

With this terminology in place, we can define an $n$-exangulated category.

\begin{defn}
\label{def:n-exangulated-category}

An \emph{$n$-exangulated category} is a triple $(\CC,\BE,\fs)$, where 
$\CC$ is an additive category, 
$\BE\colon\CC^{\op}\times\CC\to \Ab$ is a biadditive functor 
and $\fs$ is an exact realisation of $\BE$, 
	if the following are satisfied.

\begin{enumerate}[label=(EA\arabic*),wide=0pt, leftmargin=55pt, labelwidth=50pt, labelsep=5pt, align=right]

    \item\label{nEA1}
    The classes of $\fs$-inflations and of $\fs$-deflations are each closed under composition. 
        
    \item\label{nEA2}
    Suppose there are $\delta\in\BE(D,A)$ and $c\in\CC(C,D)$, with $\fs(c^{*}\delta)=[X^{\combul}]$ and $\fs(\delta)=[Y^{\combul}]$ for some $X^{\combul},Y^{\combul}\in\com_{\CC}^{n}$. 
	Then there exists a morphism 
	$f^{\combul}=(1_{A},f^{1},\ldots,f^{n},c)\colon X^{\combul}\to Y^{\combul}$ 
	realising $(1_{A},c) \colon c^{*}\delta \to \delta$, i.e.\ a commutative diagram 
	\begin{center}
	$
	\begin{tikzcd}
	c^{*}\delta \arrow{d}[swap]{(1_{A},c)} & A = X^{0} \arrow{r}{d_{X}^{0}} \arrow[equals]{d}& X^{1} \arrow{r}{} \arrow{d}{f^{1}}&\cdots\arrow{r}& X^{n} \arrow{r}{} \arrow{d}{f^{n}}& X^{n+1} = C\phantom{,} \arrow[]{d}{c}\\
	\delta & A = Y^{0} \arrow{r}{} & Y^{1} \arrow{r}{} &\cdots\arrow{r}& Y^{n} \arrow{r}{} & Y^{n+1} = D,
	\end{tikzcd}
	$
	\end{center}
	such that $\fs((d_{X}^{0})_{*}\delta)=[ \cone(\wh{f})^{\combul}]$, where $\wh{f}^{\combul}=(f^{1},\ldots,f^{n},c)$. 
    
    \item[(EA$2)^{\op}$]\label{nEA2op}
    Dual of~\ref{nEA2}.
    
\end{enumerate}

\end{defn}

If $(\CC,\BE,\fs)$ is an $n$-exangulated category and $\fs(\delta)=[X^{\combul}]$ for some extension $\delta\in\BE(C,A)$ and $X^{\combul}\in\com_{(A,C)}^{n}$, then $X^{\combul}$ is said to be an \emph{$\fs$-conflation} and $\langle X^{\combul},\delta\rangle$ is called an \emph{$\fs$-distinguished} $n$-exangle.

We close this subsection by recalling how an $(n+2)$-angulated category carries an $n$-exangulated structure.

\begin{example}
\label{example:n+2-angulated-category-is-n-exangulated}

Suppose $(\CC,\sus^{n},\pentago)$ is an $(n+2)$-angulated category (in the sense of \cite{GeissKellerOppermann-n-angulated-categories}). Thus, $\sus^{n}$ is an automorphism of $\CC$ and $\pentago$ is a collection of $(n+2)$-angles. Using $\sus^{n}$, we can define an additive bifunctor $\BE\colon \CC^{\op}\times\CC\to\Ab$ as follows. 
For objects $X^{0},X^{n+1}\in\CC$ put $\BE(X^{n+1},X^{0})\deff \CC(X^{n+1},\sus^{n} X^{0})$ 
and, for any $f\colon Y^{n+1}\to X^{n+1}$ and $g\colon X^{0}\to Y^{0}$, 
the morphism $\BE(f,g)\colon \BE(X^{n+1},X^{0}) \to \BE(Y^{n+1},Y^{0})$ is given by $\BE(f,g)(\delta)\deff (\sus^{n} g)\circ \delta\circ f$. 

Now let $\delta\in\BE(X^{n+1},X^{0})$ be some extension. Since $\CC$ is an $(n+2)$-angulated category, there is an $(n+2)$-angle of the form 
$X^{0} \to \cdots \to X^{n+1} \overset{\delta}{\to} \sus^{n} X^{0}$. 
Setting  
$\fs(\delta) = [X^{0}\to\cdots\to X^{n+1}]$, 
it can then be checked that this assignment gives an exact realisation of $\BE$. 
Moreover, it can also be shown that $(\CC,\BE,\fs)$ is an $n$-exangulated category; see \cite[Subsec.\ 4.2]{HerschendLiuNakaoka-n-exangulated-categories-I-definitions-and-fundamental-properties} for more details.
\end{example}


\subsection{An \texorpdfstring{$n$}{n}-exangulated structure induced by a subcategory}
\label{sec:induced-exangulated-structure-by-subcategory}

Suppose $(\CC, \BE,\fs)$ is an $n$-exangulated category and let $\DD$ be a full subcategory of $\CC$. 
We first recall how to obtain an $n$-exangulated category $(\CC, \BE_{\DD},\fs_{\DD})$ using the relative theory developed in \cite{HerschendLiuNakaoka-n-exangulated-categories-I-definitions-and-fundamental-properties}. 
Our main goal in this subsection is to show that $(\CC, \BE_{\DD},\fs_{\DD})$ is an $n$-exangulated subcategory (in the sense of \cite{Haugland-the-grothendieck-group-of-an-n-exangulated-category}) of $(\CC, \BE,\fs)$.

\begin{defn}

\cite[Def.\ 3.7, Def.\ 3.10]{HerschendLiuNakaoka-n-exangulated-categories-I-definitions-and-fundamental-properties} 
Denote by $\SET$ the category of sets. 
Suppose $\BF \colon \CC^{\op}\times\CC \to \SET$ is a functor. 

\begin{enumerate}[label = (\roman*)]

	\item If, for all $A,A',C,C'\in\CC$, $a\in\CC(A,A')$ and $c\in\CC(C',C)$, 
	we have $\BF(C,A)\sse\BE(C,A)$ and $\BF(c,a) = \restr{\BE(c,a)}{\BF(C,A)}$, 
	then $\BF$ is called a \emph{subfunctor of $\BE$}. 
	In this case, we write $\BF\sse\BE$. 
	
	\item For any subfunctor $\BF$ of $\BE$, if $\BF(C,A)$ is a subgroup of $\BE(C,A)$ for all $A,C\in\CC$, then we say $\BF$ is an \emph{additive subfunctor of $\BE$}. Note that, in this case, we deduce $\BF \colon \CC^{\op}\times\CC \to \Ab$ is a biadditive functor.
	
	\item Suppose $\BF$ is an additive subfunctor of $\BE$, and let $\restr{\fs}{\BF}$ denote the restriction of $\fs$ to $\BF$. We say that $\BF$ is \emph{closed on the right} if, for every $\restr{\fs}{\BF}$-conflation $X^{\combul}$, and every $Z\in\CC$,  
\begin{center}
\vspace{-8pt plus 1pt minus 1pt}
$
\begin{tikzcd}[column sep = 1.5cm]
\BF(Z,X^{0}) \arrow{r}{\BF(Z,d_{X}^{0})} & \BF(Z,X^{1}) \arrow{r}{\BF(Z,d_{X}^{1})} & \BF(Z,X^{2})
\end{tikzcd}
$
\end{center}
is exact. Dually, one defines what is meant by $\BF$ is \emph{closed on the left}. 
	
\end{enumerate}

\end{defn}

It has been shown that being closed on the right is equivalent to being closed on the left for an additive subfunctor $\BF$ of $\BE$; see \cite[Lem.\ 3.15]{HerschendLiuNakaoka-n-exangulated-categories-I-definitions-and-fundamental-properties}. Therefore, $\BF$ is simply said to be \emph{closed} if it is closed on one side.

The definition of $\BE_{\DD}$ in the next proposition comes from \cite[Def.\ 3.18]{HerschendLiuNakaoka-n-exangulated-categories-I-definitions-and-fundamental-properties}, and the proof of the result itself follows from 
\cite[Sec.\ 3]{HerschendLiuNakaoka-n-exangulated-categories-I-definitions-and-fundamental-properties}.

\begin{prop}
\label{prop:d-exangulated-structure-induced-by-subcategory}

The assignment $\BE_{\DD}$ from $\CC^{\op}\times\CC$ to $\SET$ given by 
\[
\BE_{\DD}(X^{n+1},X^{0}) \deff 
\Set{ \delta\in\BE(X^{n+1},X^{0}) | (\delta_{\sharp})_{D} = 0 \text{ for every } D\in\DD} 
\]
is a closed subfunctor of $\BE$. 
The restriction $\fs_{\DD} \deff \restr{\fs}{\BE_{\DD}}$ of the realisation $\fs$ to $\BE_{\DD}$ is an exact realisation of $\BE_{\DD}$. 
Moreover, the triple $(\CC,\BE_{\DD},\fs_{\DD})$ is an $n$-exangulated category. 

\end{prop}

\begin{proof}

The first two statements follow from 
\cite[Prop.\ 3.19]{HerschendLiuNakaoka-n-exangulated-categories-I-definitions-and-fundamental-properties} and 
\cite[Claim 3.9]{HerschendLiuNakaoka-n-exangulated-categories-I-definitions-and-fundamental-properties}, respectively. 
Then \cite[Prop.\ 3.16]{HerschendLiuNakaoka-n-exangulated-categories-I-definitions-and-fundamental-properties} implies that $(\CC,\BE_{\DD},\fs_{\DD})$ is an $n$-exangulated category. 
\end{proof}

The next notion captures what it means for a functor to preserve an $n$-exangulated structure. It will allow us to show that the $n$-exangulated structure induced by $\DD$ is compatible with the original $n$-exangulated structure of $(\CC,\BE,\fs)$.
For a covariant functor $\SF\colon \CC\to\CC'$, we denote by 
$\SF^{\op}\colon \CC^{\op}\to (\CC')^{\op}$ the induced \emph{opposite} functor.

\begin{defn}
\label{def:n-exangulated-functor}

\cite[Def.\ 2.32]{Bennett-TennenhausShah-transport-of-structure-in-higher-homological-algebra}
Suppose $(\CC',\BE',\fs')$ is also an $n$-exangulated category. 
An \emph{$n$-exangulated} functor from $\CC$ to $\CC'$ is a pair $(\SF,\Gamma)$, 
where $\SF\colon \CC\to\CC'$ is an additive covariant functor and 
\[
\Gamma=\{\Gamma_{(C,A)}\}_{(C,A)\in\CC^{\op}\times\CC}\colon \BE \Longrightarrow \BE'(\SF^{\op}-,\SF-)
\]
is a natural transformation of functors from $\CC^{\op}\times\CC$ to $\Ab$, such that 
$\fs'(\Gamma_{(X^{n+1},X^{0})}(\delta))=[\SF_{\com} X^{\combul}]$ whenever $\fs(\delta)=[X^{\combul}]$. 
\end{defn}

\begin{defn}

\cite[Def.\ 3.7]{Haugland-the-grothendieck-group-of-an-n-exangulated-category}
Suppose $(\CS,\BE',\fs')$ is an $n$-exangulated category, where $\CS$ is a full subcategory of $\CC$ that is closed under isomorphisms, and let $\iota\colon \CS \to \CC$ be the inclusion functor. 
Then $(\CS,\BE',\fs')$ is called an \emph{$n$-exangulated subcategory} of $(\CC,\BE,\fs)$ 
if there is an $n$-exangulated functor $(\iota, \Gamma)$, where 
$\Gamma_{(C,A)}\colon \BE'(C,A) \to \BE(C,A)$ is the inclusion of abelian groups for all $A,C\in \CS$.

\end{defn}

We are now in a position to give the main result of this subsection.

\begin{thm}
\label{thm:subcategory-induces-d-exangulated-subcategory-of-d-exangulated-category}

Let $\DD$ be a full subcategory of $\CC$. 
The triple $(\CC,\BE_{\DD},\fs_{\DD})$ is an $n$-exangulated subcategory of $(\CC,\BE,\fs)$. 

\end{thm}

\begin{proof}

Since $\BE_{\DD}\sse\BE$ and $\fs_{\DD}$ is a restriction of $\fs$, it is clear that the identity functor $\idfunc{\CC}$, with the natural transformation $\Gamma\colon \BE_{\DD} \Rightarrow \BE(\idfunc{\CC}^{\op}-,\idfunc{\CC}-) = \BE$ given by $\Gamma_{(C,A)}(\delta) = \delta$ for each $A,C\in\CC$ and each $\delta\in\BE_{\DD}(C,A)$, is an $n$-exangulated functor.
\end{proof}

\begin{rem}
\label{rem:induced-subfunctor-when-D-additive-subcategory-of-n-angulated-category}

Suppose $(\CC,\sus^{n},\pentago)$ is an $(n+2)$-angulated category, and write $(\CC,\BE,\fs)$ for its $n$-exangulated structure; see Example~\ref{example:n+2-angulated-category-is-n-exangulated} above. 
Let $\DD$ be an additive subcategory of $\CC$. 
It is immediate that, for any $\delta\in\CC(X^{n+1},\sus^{n}X^{0})$, 
we have $\delta\in\BE_{\DD}(X^{n+1},X^{0})$ if and only if $F_{\DD}(\delta)=0$ (see \ref{1-1-vi} in Subsection~\ref{sec:conventions-notation}), i.e.\
\[
\BE_{\DD}(X^{n+1},X^{0}) 
		= \Set { \delta \in \CC(X^{n+1},\sus^{d}X^{0}) | F_{\DD}(\delta) = 0 }.
\]
Furthermore, since $\DD^{\perp_{0}}$ is closed under finite direct sums, we may consider the ideal $[\DD^{\perp_{0}}]$ of $\CC$. 
The subfunctor $\BE_{\DD}$ of $\BE$ satisfies
\(
[\DD^{\perp_{0}}](X^{n+1},\sus^{n}X^{0}) \subseteq \BE_{\DD}(X^{n+1},X^{0}).
\) 

We note that if $n=1$ (so that $\CC$ is a triangulated category), $\DD$ is extension-closed in $\CC$ and each object of $\CC$ admits a minimal right $\DD$-approximation, then one can show this inclusion is an equality using the triangulated Wakamatsu Lemma \cite[Lem.~1.1]{BuanMarsh-BM2}.
\end{rem}


\subsection{The Grothendieck group of an \texorpdfstring{$n$}{n}-exangulated category}
\label{sec:grothendieck-groups-of-n-exangulated-categories}

In classical settings, the notion of a Grothendieck group has been used previously to give connections between certain subcategories of a category and subgroups of the ambient category's Grothendieck group; see \cite{Matsui-classifying-dense-resolving-and-coresolving-subcategories-of-exact-categories-via-grothendieck-groups}, \cite{Thomason-the-classification-of-triangulated-subcategories}. More recently, some of these results have been extended to the $n$-exangulated cases; see \cite{BerghThaule-the-grothendieck-group-of-an-n-angulated-category}, \cite{Haugland-the-grothendieck-group-of-an-n-exangulated-category}, \cite{ZhuZhuang-grothendieck-groups-in-extriangulated-categories}. In contrast, we will look at how the Grothendieck group theory changes as we tweak the $n$-exangulated structure on a fixed category and we will also be considering certain quotients of split Grothendieck groups. 

The first construction we need is well-known; see, for example \cite[Ch.\ VII]{Bass-algebraic-k-theory-book}.

\begin{defn}

Let $\CA$ be a skeletally small additive category. 
An isomorphism class of an object $A\in\CA$ will be denoted as usual by $[A]$. 
Denote by $\CF(\CA)$ the free abelian group on the isomorphism classes of objects in $\CA$. 
The \emph{split Grothendieck group} $K_{0}^{\sp}(\CA)$ of $\CA$ is the quotient of $\CF(\CA)$ by the subgroup $\Braket{[A] - [A\oplus B] + [B] | A,B\in\CA}$. 

\end{defn}

Since we will be working modulo relations coming from $n$-exangles, it is convenient to introduce the following notation.

\begin{defn}
\label{def:euler-relation}

Let $\CA$ be a skeletally small additive category. 
For a complex 
\[
X^{\combul}\colon 
\begin{tikzcd}[column sep=0.5cm]
X^{0} \arrow{r}&X^{1} \arrow{r}& \cdots\arrow{r}& X^{n}\arrow{r} & X^{n+1}
\end{tikzcd}
\] 
in $\com_{\CA}^{n}$, we define the associated \emph{Euler relation} to be the alternating sum
\[
\chi(X^{\combul}) 
	\deff \sum_{i=0}^{n+1} (-1)^{i} [X^{i}] 
	= [X^{0}] - [X^{1}] + \cdots + (-1)^{n}[X^{n}] + (-1)^{n+1}[X^{n+1}],
\]
of isomorphism classes, viewed as an element of $\CF(\CA)$.
\end{defn}

\begin{rem}
\label{rem:reversing-order-of-euler-relation}

For the convenience of the reader, 
we make the following observation that we use several times in the remainder of the article: 
\begin{align*}
(-1)^{n+1}\chi(X^{\combul}) 	
	&= [X^{n+1}] - [X^{n}] + \cdots + (-1)^{n}[X^{1}] + (-1)^{n+1}[X^{0}] \\
	&= \sum_{i=0}^{n+1} (-1)^{i} [X^{n+1-i}].
\end{align*}
\end{rem}

We can now state the main definition of this subsection.

\begin{defn}
\label{def:grothendieck-group-of-n-exangulated-category}

\cite[Def.\ 4.1]{Haugland-the-grothendieck-group-of-an-n-exangulated-category}
Let $n\geq 1$ be a positive integer. 
Let $(\CC,\BE,\fs)$ be a skeletally small $n$-exangulated category. 
As before, let $\CF(\CC)$ be the free abelian group on the isomorphism classes of objects in $\CC$. 
Let $\CG(\CC,\BE,\fs)$ denote the subgroup of $\CF(\CC)$ generated by the subset
\[
\begin{cases}
\Set{ \chi(X^{\combul}) | X^{\combul} \text{ is an } \fs\text{-conflation in } (\CC,\BE,\fs) } & \text{if } n \text{ is odd}\\
\Set{ \chi(X^{\combul}) | X^{\combul} \text{ is an } \fs\text{-conflation in } (\CC,\BE,\fs) } \cup  \set{ [0] } & \text{if } n \text{ is even.}\\
\end{cases}
\]
The \emph{Grothendieck group} of $(\CC,\BE,\fs)$ is the quotient 
$K_{0}(\CC,\BE,\fs)\deff\CF(\CC)/\CG(\CC,\BE,\fs)$. 
\end{defn}

Let $(\CC,\BE,\fs)$ be an $n$-exangulated category and let $\DD\sse \CC$ be a full subcategory. We conclude this subsection by giving a description of the Grothendieck group of $(\CC,\BE_{\DD},\fs_{\DD})$ as a quotient of the split Grothendieck group $K_{0}^{\sp}(\CC)$.

\begin{defn}
\label{def:subgroup-I-sub-D}

Define the subgroup 
\[
\CI_{\DD}\deff \Braket{
	\chi(X^{\combul}) | X^{\combul} \text{ is an } \fs_{\DD}\text{-conflation in } (\CC,\BE_{\DD},\fs_{\DD})
}
\]
of the split Grothendieck group $K_{0}^{\sp}(\CC)$. Then we have $K_{0}(\CC,\BE_{\DD},\fs_{\DD}) \iso K_{0}^{\sp}(\CC)/\CI_{\DD}$, and we identify these groups. 
Furthermore, let us denote by $Q_{\DD}$ the canonical surjection 
$K_{0}^{\sp}(\CC)\to K_{0}^{\sp}(\CC) / \CI_{\DD}$. 

\end{defn}

There are the two extremal cases, which we quickly discuss now. 

\begin{enumerate}[label = (\alph*)]

	\item If $\DD = 0$, then $\BE_{\DD} = \BE_{0} = \BE$ and so $(\CC,\BE_{0},\fs_{0}) = (\CC,\BE,\fs)$. In addition, $\CI_{\DD} = \CI_{0}$ is generated by $\chi(X^{\combul})$ for each $\fs$-conflation $X^{\combul}$ in $(\CC,\BE,\fs)$. 

	\item If $\DD = \CC$, then $\BE_{\DD}(X^{n+1},X^{0}) = 0$ for any $X^{0},X^{n+1}\in\CC$, and so $\CI_{\DD} = \CI_{\CC}$ is generated by the Euler relations associated to the $\fs$-conflations of the form 
	\begin{center}
	\begin{tikzcd}[column sep=0.7cm]
	0 \arrow{r}& \cdots \arrow{r}& 0 \arrow{r}& A \arrow{r}{1_{A}}& A \arrow{r}& 0 \arrow{r}& \cdots\arrow{r} & 0,
	\end{tikzcd}
	\end{center}
	where $A\in\CC$ and the identity morphism $1_{A}$ may be in any position; see \cite[Prop.\ 2.14]{Haugland-the-grothendieck-group-of-an-n-exangulated-category}. Thus, $\CI_{\CC} = \lan [0]\ran$ is the trivial subgroup of $K_{0}^{\sp}(\CC)$, and hence $K_{0}^{\sp}(\CC)/\CI_{\CC} = K_{0}^{\sp}(\CC)$.

\end{enumerate}


\section{Cluster tilting subcategories and the \texorpdfstring{$(\lowercase{d}+2)$}{d+2}-angulated index}
\label{sec:cluster-tilting-subcategories}

For a triangulated category $\CC$ and cluster tilting subcategory $\CT\sse\CC$, a triangulated index was introduced by Palu in \cite{Palu-cluster-characters-for-2-calabi-yau-triangulated-categories}, giving a way to associate an element of the split Grothendieck group $K_{0}^{\sp}(\CT)$ to an object of $\CC$. A higher analogue was subsequently introduced in \cite{Jorgensen-tropical-friezes-and-the-index-in-higher-homological-algebra}; see Definition~\ref{def:d+2-angulated-index-wrt-T}. 
We show that this gives an isomorphism between the split Grothendieck group of $\CT$ and the Grothendieck group of the $d$-exangulated structure on $\CC$ induced by $\CT$.

First, we recall some essential definitions.

\begin{defn}
\label{def:n-cluster-tilting}

\cite{Iyama-maximal-orthogonal-subcategories-of-triangulated-categories-satisfying-serre-duality}, 
\cite[Sec.\ 3]{IyamaYoshino-mutation-in-tri-cats-rigid-CM-mods}
Let $n\geq 2$ be a positive integer. 
Let $\CC$ be an idempotent complete, triangulated category. 
Suppose $\CT$ is an additive subcategory of $\CC$. 
We say that $\CT$ is an \emph{$n$-cluster tilting subcategory of $\CC$} if: 

\begin{enumerate}[label = (\roman*)]

	\item $\CT$ is functorially finite; and

	\item 
	$\begin{aligned}[t]
		\CT 	&= \Set{X\in \CC | \Ext_{\CC}^{i}(\CT,X) = 0 \text{ for all } 1 \leq i \leq n-1} \\
		 	&= \Set{X\in \CC | \Ext_{\CC}^{i}(X,\CT) = 0 \text{ for all } 1 \leq i \leq n-1}.
	\end{aligned}$

\end{enumerate}
\end{defn}

An interesting characterisation of $n$-cluster tilting subcategories is given in 
\cite[Thm.\ 5.3]{Beligiannis-relative-homology-higher-cluster-tilting-theory-and-categorified-Auslander-Iyama-correspondence}.

\begin{defn}
\label{def:Oppermann-Thomas-cluster-tilting}

\cite[Def.\ 6.2]{Fedele-grothendieck-groups-of-triangulated-categories-via-cluster-tilting-subcategories}, 
\cite[Def.\ 5.3]{OppermannThomas-higher-dimensional-cluster-combinatorics-and-representation-theory} 
Let $d\geq 1$ be a positive integer. 
Let $\CS$ be a $(d+2)$-angulated category. 
Suppose $\CT$ is an additive subcategory of $\CS$. 
We say that $\CT$ is an \emph{Oppermann-Thomas cluster tilting subcategory of $\CS$} if: 

\begin{enumerate}[label = (\roman*)]
	
	\item $\CS(\CT,\sus^{d}\CT) = 0$; and 
	
	\item for each $S\in\CS$ there is a $(d+2)$-angle 
	\begin{tikzcd}[column sep=0.5cm]
	T^{0} \arrow{r} &\cdots \arrow{r} &T^{d} \arrow{r} &S \arrow{r} &\sus^{d}T^{0}
	\end{tikzcd}
	in $\CS$ with $T^{i}\in\CT$ for all $0 \leq i \leq d$.

\end{enumerate}

\end{defn}

Note that (i) and (ii) in Definition~\ref{def:Oppermann-Thomas-cluster-tilting} imply $\CT$ is functorially finite.

\begin{setup}
\label{setup:initial-d-setup-C-T-S}

For the remainder of Sections~\ref{sec:cluster-tilting-subcategories} and \ref{sec:K0TmodNX-as-a-grothendieck-group}, we fix the following. Let $\field$ be a field, and let $d,n$ be positive integers with $n=2d$. We assume $\CC$ is a skeletally small, $\field$-linear, Krull-Schmidt, $\Hom$-finite triangulated category with suspension functor $\sus$. We suppose there are subcategories $\CT\sse\CS$ in $\CC$, where $\CT$ is $n$-cluster tilting in $\CC$, and $\CS$ is $d$-cluster tilting in $\CC$ and $\sus^{d}\CS\sse\CS$.
\end{setup}

\begin{rem}
\label{rem:on-setup-3-3}

\begin{enumerate}[label=(\roman*)]

	\item As noted in \cite[Rem.\ 5.2]{Jorgensen-tropical-friezes-and-the-index-in-higher-homological-algebra}, Setup~\ref{setup:initial-d-setup-C-T-S} implies:
$\CS$ has the structure of a $(d + 2)$-angulated category following the ``standard construction'' given by Geiss--Keller--Oppermann (see \cite[Thm.\ 1]{GeissKellerOppermann-n-angulated-categories}), with $d$-suspension functor $\sus^{d}$, which is the restriction to $\CS$ of the $d$'th power of $\sus$; and that $\CT$ is an Oppermann-Thomas cluster tilting subcategory of $\CS$ by \cite[Thm.\ 5.25]{OppermannThomas-higher-dimensional-cluster-combinatorics-and-representation-theory}.

	\item We denote by $(\CS,\BE,\fs)$ the $d$-exangulated structure of $\CS$ (see Example~\ref{example:n+2-angulated-category-is-n-exangulated}). 

	\item 
	\label{rem-3-3-iii}
	The subcategory $\CT$ induces a $d$-exangulated structure on $\CS$ as in Proposition~\ref{prop:d-exangulated-structure-induced-by-subcategory}, which we denote by $(\CS,\BE_{\CT},\fs_{\CT})$. Furthermore, $(\CS,\BE_{\CT},\fs_{\CT})$ is a $d$-exangulated subcategory of $(\CS,\BE,\fs)$ by Theorem~\ref{thm:subcategory-induces-d-exangulated-subcategory-of-d-exangulated-category}.
	Setup~\ref{setup:initial-d-setup-C-T-S} implies $\CT$ itself is an additive subcategory of $\CS$, so, as in Remark~\ref{rem:induced-subfunctor-when-D-additive-subcategory-of-n-angulated-category}, we have

	\begin{center}
$\BE_{\CT}(X^{d+1},X^{0}) 
		= \Set { \delta \in \CS(X^{d+1},\sus^{d}X^{0}) | F_{\CT}(\delta) = 0 }$.
	\end{center}
\end{enumerate}

\end{rem}

With this Setup, we can then recall the higher index with respect to $\CT$.

\begin{defn}
\label{def:d+2-angulated-index-wrt-T}

\cite[Def.\ B]{Jorgensen-tropical-friezes-and-the-index-in-higher-homological-algebra} 
For $S\in\CS$, let 
\begin{equation}\label{eqn:dplus2-angle-for-index-wrt-T-of-object-S}
\begin{tikzcd}
T^{0}_{S} \arrow{r}{\tau^{0}}& \cdots\arrow{r}{\tau^{d-1}}& T^{d}_{S}\arrow{r} & S\arrow{r} & \sus^{d} T^{0}_{S}
\end{tikzcd} 
\end{equation}
be a $(d+2)$-angle in $\CS$ with $T^{i}_{S}\in\CT$ and $\tau^{i}$ in the radical of $\CS$ for each $i$. 
The \emph{$(d+2)$-angulated index of $S$ with respect to $\CT$} is 
\[
\dindx{\CT}(S) 
	= \sum_{i=0}^{d}(-1)^{i}[T^{d-i}_{S}]
	= [T^{d}_{S}] - [T^{d-1}_{S}] + \cdots + (-1)^{d-1}[T^{1}_{S}] + (-1)^{d}[T^{0}_{S}],
\]
viewed as an element of $K_{0}^{\sp}(\CT)$. 

\end{defn}

\begin{rem}
\label{rem:remarks-on-dplus2-angulated-index}

We note the following. 

\begin{enumerate}[label=(\roman*)]

	\item Let $T_{S}^{\combul}$ denote the $\fs$-conflation corresponding to the $(d+2)$-angle \eqref{eqn:dplus2-angle-for-index-wrt-T-of-object-S}, i.e.\ $T_{S}^{\combul}$ is the complex 
	$\begin{tikzcd}
T^{0}_{S} \arrow{r}{\tau^{0}}& \cdots\arrow{r}{\tau^{d-1}}& T^{d}_{S}\arrow{r} & S
\end{tikzcd} $
in $\com_{\CS}^{d}$. 
Then notice that, as elements of $K_{0}^{\sp}(\CS)$, we have the equality $\dindx{\CT}(S) = [S] - (-1)^{d+1}\chi(T_{S}^{\combul})$, 
using Definition~\ref{def:euler-relation} and Remark~\ref{rem:reversing-order-of-euler-relation}.

	\item For any $T\in\CT$, we have $\dindx{\CT}(T) = [T]$ as $\CS$ admits the $(d+2)$-angle 
	\begin{center}
	\begin{tikzcd}[column sep=0.6cm]
	0 \arrow{r} & \cdots \arrow{r} & 0 \arrow{r} &T \arrow{r}{1_{T}} &T \arrow{r} & \sus^{d} 0.
	\end{tikzcd}
	\end{center}
	
\end{enumerate}

\end{rem}

In order to show that the $(d+2)$-angulated index induces group homomorphisms between Grothendieck groups, we use the following result from \cite{Jorgensen-tropical-friezes-and-the-index-in-higher-homological-algebra}, which shows this index is additive on $(d+2)$-angles in $\CS$ up to an error term. 
Recall that $F_{\CT}\colon \CS \to \rmod{\CT}$ is given by $F_{\CT}(S) = \restr{\CS(-,S)}{\CT}$ on objects (see \ref{1-1-vi} in Subsection~\ref{sec:conventions-notation}).

\begin{thm}
\label{thm:Euler-sum-of-d-indices-equals-theta-of-image-of-connecting-map}

\cite[Thm.\ C]{Jorgensen-tropical-friezes-and-the-index-in-higher-homological-algebra} 
There is a homomorphism $\theta\colon K_{0}(\rmod{\CT}) \to K_{0}^{\sp}(\CT)$, such that if 
$\begin{tikzcd}[column sep=0.5cm]
S^{0} \arrow{r}& \cdots\arrow{r}& S^{d}\arrow{r} & S^{d+1}\arrow{r}{\gamma} & \sus^{d} S^{0}
\end{tikzcd}$
is a $(d+2)$-angle in $\CS$, then 
\[
\sum_{i=0}^{d+1} (-1)^{i} \dindx{\CT}(S^{d+1-i}) = \theta([\Im F_{\CT}(\gamma)]).
\]

\end{thm}

The first homomorphism obtained via the higher index is from the split Grothendieck group of $\CS$ to that of $\CT$.

\begin{prop}
\label{prop:d-index-gives-hom-split-G-group-of-S-to-split-G-group-of-T}

The $(d+2)$-angulated index $\dindx{\CT}$ induces a group homomorphism $K_{0}^{\sp}(\CS) \to K_{0}^{\sp}(\CT)$, which we also denote by $\dindx{\CT}$. 

\end{prop}

\begin{proof}

By \cite[Rem.\ 5.4]{Jorgensen-tropical-friezes-and-the-index-in-higher-homological-algebra}, $\dindx{\CT}(S)$ is well-defined and depends only on the isomorphism class of $S$ in $\CS$. 
Therefore, by extending linearly, it induces a group homomorphism $\CF(\CS) \to K_{0}^{\sp}(\CT)$. 
Let $A,B\in\CS$ be arbitrary. 
Then $\CS$ admits the $(d+2)$-angle 
\[
\begin{tikzcd}[ampersand replacement=\&, column sep = 1.2cm]
0 \arrow{r} \& \cdots \arrow{r} \&0 \arrow{r} \&A \arrow{r}{
	\begin{psmallmatrix} 1_{A} \\ 0
	\end{psmallmatrix} }
\&A\oplus B \arrow{r}{
	\begin{psmallmatrix} 0 & 1_{B}
	\end{psmallmatrix} }
\& B \arrow{r}{0} \& \sus^{d}0,
\end{tikzcd}
\]
which yields 
$\dindx{\CT}(B) - \dindx{\CT}(A\oplus B) + \dindx{\CT}(A) = \theta([\Im 0]) = 0$ by Theorem~\ref{thm:Euler-sum-of-d-indices-equals-theta-of-image-of-connecting-map}. 
Hence, we obtain a well-defined group homomorphism $\dindx{\CT}\colon K_{0}^{\sp}(\CS) \to K_{0}^{\sp}(\CT)$.
\end{proof}

\begin{notation}

Suppose $\DD$ is a full subcategory of the $(d+2)$-angulated category $\CS$. 
Recall that there is, thus, a $d$-exangulated subcategory $(\CS,\BE_{\DD},\fs_{\DD})$ of $(\CS,\BE,\fs)$ by Theorem~\ref{thm:subcategory-induces-d-exangulated-subcategory-of-d-exangulated-category}. 
Let $\delta\in\BE_{\DD}(X^{d+1},X^{0})\sse\CS(X^{d+1},\sus^{d}X^{0})$. 
If $\fs_{\DD}(\delta) = [X^{0}\to\cdots\to X^{d+1}]$ such that  
$X^{0}\to\cdots\to X^{d+1}$ completes to a $(d+2)$-angle $X^{0}\to\cdots\to X^{d+1} \to \sus^{d}X^{0}$ in $\CS$, then we emphasise this by depicting the $\fs_{\DD}$-distinguished $d$-exangle $\lan X^{\combul},\delta\ran$ as 
\[
\begin{tikzcd}
X^{0} \arrow{r}&X^{1} \arrow{r}& \cdots\arrow{r}& X^{d}\arrow{r} & X^{d+1} \arrow[dashed]{r}{\delta}& \sus^{d} X^{0},
\end{tikzcd}
\] 
to simultaneously indicate that $\lan X^{\combul},\delta\ran$ is an $\fs_{\DD}$-distinguished $d$-exangle and that the whole diagram is a $(d+2)$-angle in $\CS$. 
Although this notation is used for exangulated categories more generally, we will reserve this notation specifically for the type of situation just described. 
\end{notation}

The authors thank P.-G. Plamondon for bringing \cite[Prop.\ 4.11]{PadrolPaluPilaudPlamondon-associahedra-for-finite-type-cluster-algebras-and-minimal-relations-between-g-vectors} to the attention of the first author. Our main result of this section is the following higher version of \cite[Prop.\ 4.11]{PadrolPaluPilaudPlamondon-associahedra-for-finite-type-cluster-algebras-and-minimal-relations-between-g-vectors}, which was the starting point for the present article.

\begin{thm}
\label{thm:d-index-gives-isomorphism-G-group-of-T-induced-exangulated-structure-on-S-to-split-G-group-of-T}

The $(d+2)$-angulated index $\dindx{\CT}$ induces an isomorphism $K_{0}(\CS,\BE_{\CT},\fs_{\CT}) \to K_{0}^{\sp}(\CT)$, which we also denote by $\dindx{\CT}$.  The inverse $L\colon K_{0}^{\sp}(\CT) \to K_{0}(\CS,\BE_{\CT},\fs_{\CT})$ is defined on generators by $L([T]) = [T]$ for each $T\in\CT$.

\end{thm}

\begin{proof} 

Recall that $K_{0}(\CS,\BE_{\CT},\fs_{\CT}) = K_{0}^{\sp}(\CS)/\CI_{\CT}$; 
see Definition~\ref{def:subgroup-I-sub-D}. 
We first show $\CI_{\CT}$ is contained in the kernel of $\dindx{\CT}\colon K_{0}^{\sp}(\CS) \to K_{0}^{\sp}(\CT)$ from Proposition~\ref{prop:d-index-gives-hom-split-G-group-of-S-to-split-G-group-of-T}.  
Thus, let 
\[
\begin{tikzcd}
S^{0} \arrow{r}& \cdots\arrow{r}& S^{d}\arrow{r} & S^{d+1}\arrow[dashed]{r}{\gamma} & \sus^{d} S^{0}
\end{tikzcd}
\]
be an $\fs_{\CT}$-distinguished $d$-exangle in $(\CS,\BE_{\CT},\fs_{\CT})$. 
Thus, $F_{\CT}(\gamma) = 0$ by Remark~\ref{rem:on-setup-3-3}\ref{rem-3-3-iii}, 
and we have 
\begin{align*}
\sum_{i=0}^{d+1} (-1)^{i} \dindx{\CT}(S^{i}) &= (-1)^{d+1}\sum_{i=0}^{d+1} (-1)^{i} \dindx{\CT}(S^{d+1-i}) 
&&\text{(see Remark~\ref{rem:reversing-order-of-euler-relation})}\\
&= (-1)^{d+1}\theta([\Im F_{\CT}(\gamma)]) 
&&\text{by Theorem~\ref{thm:Euler-sum-of-d-indices-equals-theta-of-image-of-connecting-map}}\\
&= 0.
\end{align*}
Hence, $\dindx{\CT}$ induces a group homomorphism
$K_{0}(\CS,\BE_{\CT},\fs_{\CT}) \to K_{0}^{\sp}(\CT)$, which we again denote by $\dindx{\CT}$. 

It remains to show that this is an isomorphism, and we do this by constructing an inverse of $\dindx{\CT}$. 
Define a group homomorphism $L\colon K_{0}^{\sp}(\CT) \to K_{0}(\CS,\BE_{\CT},\fs_{\CT})$ on generators by $L([T]) = [T]$ for each $T\in\CT$ and extend linearly. 
By Remark~\ref{rem:remarks-on-dplus2-angulated-index}(ii), we see that $\dindx{\CT}\circ L$ is the identity on $ K_{0}^{\sp}(\CT)$. 
On the other hand, for $S\in\CS$, there is a $(d+2)$-angle of the form \eqref{eqn:dplus2-angle-for-index-wrt-T-of-object-S} as $\CT$ is Oppermann-Thomas cluster tilting. 
By dropping trivial summands from the morphisms $\tau^{i}$, one can assume each $\tau^{i}$ is radical. 
As $\CT$ is Oppermann-Thomas cluster tilting, we have $\CS(\CT,\sus^{d}\CT) = 0$. 
Thus, in $ K_{0}(\CS,\BE_{\CT},\fs_{\CT})$, we see 
\begin{align*}
[S] 	&=  - \sum_{i=1}^{d+1} (-1)^{i}[T^{d+1-i}] \\
	&= \sum_{i=1}^{d+1} (-1)^{i-1}[T^{d+1-i}] \\
	&= L\left(\sum_{i=0}^{d} (-1)^{i}[T^{d-i}]\right) \\
	&= L(\dindx{\CT}([S])).
\end{align*}
This shows that $L \circ \dindx{\CT}$ is the identity on $K_{0}(\CS,\BE_{\CT},\fs_{\CT})$, and hence $\dindx{\CT}\colon  K_{0}(\CS,\BE_{\CT},\fs_{\CT}) \to K_{0}^{\sp}(\CT)$ is indeed an isomorphism. 
\end{proof}


\section{The quotient \texorpdfstring{$K_{0}^{\sp}(\CT)/N_{\CX}$}{K0TmodNX} as a Grothendieck group}
\label{sec:K0TmodNX-as-a-grothendieck-group}

Throughout this section, we still have the base assumptions laid out in Setup~\ref{setup:initial-d-setup-C-T-S}. 
In the coming subsections, we specialise a further three times as detailed below in order to prove our main result, namely, Theorem~\ref{thm:G-R-is-an-isomorphism}. 
Relevant definitions are recalled in the corresponding subsection.

\begin{setup}[for Subsection~\ref{sec:mutation-pairs}]
\label{setup:all-indecomposable-objects-of-T-not-in-X-are-part-of-exchange-pair}

In addition to Setup~\ref{setup:initial-d-setup-C-T-S}, suppose: 
$\CS$ is \emph{$2d$-Calabi-Yau} (see Definition~\ref{def:2d-calabi-yau}); 
$\CX$ is a functorially finite, additive subcategory of $\CT$; 
and all objects in $\ind\CT\setminus\ind\CX$ have a \emph{mutation} (see Definition~\ref{def:mutation-pair}).

\end{setup}

\begin{setup}[for Subsection~\ref{sec:formula-for-theta-on-simples}]
\label{setup:assume-Im-Fgamma-is-simple} 

In addition to Setups~\ref{setup:initial-d-setup-C-T-S} and \ref{setup:all-indecomposable-objects-of-T-not-in-X-are-part-of-exchange-pair}, we assume that, for each $T\in\ind\CT\setminus\ind\CX$, the image $\Im F_{\CT}(\gamma^{T})$ (see \ref{1-1-vi} in Subsection~\ref{sec:conventions-notation}) is isomorphic to the simple module 
$\ol{S}_{T} = \CT(-,T)/ \rad_{\CT}(-,T)$ in $\rMod{\CT}$ (see \eqref{eqn:simple-at-T}), 
where $\gamma^{T}$ is as in the exchange $(d+2)$-angle \eqref{eqn:exchange-d+2-angle-gamma-T}.

\end{setup}

\begin{setup}[for Subsection~\ref{sec:G-X-is-an-isomorphism}]
\label{setup:F-lands-in-finite-length-modules}

In addition to Setups~\ref{setup:initial-d-setup-C-T-S}, \ref{setup:all-indecomposable-objects-of-T-not-in-X-are-part-of-exchange-pair} and \ref{setup:assume-Im-Fgamma-is-simple}, we assume that for any $S\in\CS$, the functor $F_{\CT}(S) = \restr{\CS(-, S)}{\CT}$ (see \ref{1-1-vi} in Subsection~\ref{sec:conventions-notation}) has finite length in $\rmod{\CT}$, i.e.\ $F_{\CT}(S) \in \fl{\CT}$. 

\end{setup}

\begin{rem}\label{rem:setups-in-sec-4}
We note these Setups are satisfied in several cases of interest. 
\begin{enumerate}[label=(\roman*)]

\item\label{item:setup-4-1} For example, Setup~\ref{setup:all-indecomposable-objects-of-T-not-in-X-are-part-of-exchange-pair} is satisfied if 
$\CT = \add T$ for some Oppermann-Thomas cluster tilting object $T\in\CS$, and $\CX\sse\CT$ is chosen as needed for Setup~\ref{setup:all-indecomposable-objects-of-T-not-in-X-are-part-of-exchange-pair}. 
In particular, when $d>1$, we note that not \emph{any} choice of a functorially finite, additive subcategory $\CX\sse\CT$ will result in each indecomposable in $\CT\setminus\CX$ having a mutation, so one has to be careful when choosing $\CX$; see \cite[Rem.\ 5.8]{OppermannThomas-higher-dimensional-cluster-combinatorics-and-representation-theory}.

\item Setup~\ref{setup:assume-Im-Fgamma-is-simple} holds if e.g.\ $d=1$ and $\CT$ is part of a \emph{cluster structure} in the sense of 
\cite[p.\ 1039]{BuanIyamaReitenScott-cluster-structures-for-2-calabi-yau-categories-and-unipotent-groups}.

\item The additional condition in Setup~\ref{setup:F-lands-in-finite-length-modules}---namely, that $F_{\CT}(S)$ has finite length---is satisfied if, for example, $\CT$ is \emph{locally bounded}, i.e.\ $\Hom$-finite and for each indecomposable object 
		$T\in\CT$, there are only finitely many objects $T'\in\ind\CT$ such that $\CT(T,T')\neq 0$ or $\CT(T',T)\neq 0$. 
		See \cite[Rem.\ 5.4]{JorgensenPalu-a-caldero-chapoton-map-for-infinite-clusters}. 
		(Note that we are using `locally bounded' in the sense of \cite{LenzingReiten-hereditary-noetherian-categories-of-positive-euler-characteristic}, 
		and not as in \cite{DowborSkowronski-on-the-representation-type-of-locally-bounded-categories}.)
\end{enumerate}
\end{rem}

In this section we will establish the following. 

\begin{thm}
\label{thm:diagram-of-homomorphisms}

There is a commutative diagram 
\begin{equation}
\label{eqn:diagram-of-homs}
\begin{tikzcd}[column sep=1.5cm, row sep=2cm]
K_{0}^{\sp}(\CS)
\arrow[bend right=60, two heads]{dd}[swap]{\dindx{\CT}}
\arrow[two heads]{d}{Q_{\CT}}
\arrow[two heads]{dr}{Q_{\CX}}
\arrow[two heads]{r}{Q_{0}} 	
	& 	K_{0}(\CS,\BE,\fs)\\
K_{0}(\CS,\BE_{\CT},\fs_{\CT})
\arrow[bend left=20]{d}{\dindx{\CT}} 
\arrow[two heads]{r}{\wt{Q}}
\commutes[\iso]{d}	
	&	K_{0}(\CS,\BE_{\CX},\fs_{\CX})
		\arrow[two heads]{u}{}\\
K_{0}^{\sp}(\CT)
\arrow[two heads]{r}{Q_{N_{\CX}}}
\arrow[bend left=20]{u}{L} 
	&	K_{0}^{\sp}(\CT)/N_{\CX}
		\arrow[dotted]{u}{G_{\CX}}[swap]{\iso}
\end{tikzcd}
\end{equation}
of abelian groups, with the indicated surjections and isomorphisms. 
In particular, there is an isomorphism 
$K_{0}^{\sp}(\CT)/N_{\CX} \iso K_{0}(\CS,\BE_{\CX},\fs_{\CX})$.
\end{thm}

Some of the categories and homomorphisms in \eqref{eqn:diagram-of-homs} have already appeared in earlier sections, and others we introduce in this section. 
That $L$ and $\dindx{\CT}$ are mutually inverse follows from 
Theorem~\ref{thm:d-index-gives-isomorphism-G-group-of-T-induced-exangulated-structure-on-S-to-split-G-group-of-T}.

The upper half of \eqref{eqn:diagram-of-homs} exists and commutes by the following formal considerations. 
Recall that in Definition~\ref{def:subgroup-I-sub-D}, for each full subcategory $\DD$ of an $n$-exangulated category $(\CC,\BE,\fs)$, we defined a subgroup 
$\CI_{\DD}\leq K_{0}^{\sp}(\CC)$
and the canonical epimorphism
\[
Q_{\DD}\colon K_{0}^{\sp}(\CC)\to K_{0}^{\sp}(\CC) / \CI_{\DD} = K_{0}(\CC,\BE_{\DD},\fs_{\DD}). 
\]
Since $0\sse\CX\sse\CT\sse\CS$, we have that $0=\CI_{\CS}\leq\CI_{\CT}\leq\CI_{\CX}\leq\CI_{0}$ as subgroups of $ K_{0}^{\sp}(\CS)$. 
Thus, applications of the universal property of the quotient yield a commutative diagram
\begin{equation}
\label{eqn:diagram-of-Qs}
\begin{tikzcd}[column sep=1.2cm, row sep=1.5cm]
K_{0}^{\sp}(\CS) \arrow[two heads]{r}{Q_{0}}\arrow[two heads]{d}{Q_{\CT}}\arrow[two heads]{dr}{Q_{\CX}}& K_{0}(\CS,\BE,\fs) \\
K_{0}(\CS,\BE_{\CT},\fs_{\CT}) \arrow[two heads]{r}{\wt{Q}}& K_{0}(\CS,\BE_{\CX},\fs_{\CX})\arrow[two heads]{u}{}
\end{tikzcd}
\end{equation}
of surjective abelian group homomorphisms. See Subsection~\ref{sec:grothendieck-groups-of-n-exangulated-categories} for more details. 

Accordingly, the primary objective of this section is to establish the existence of the lower half of \eqref{eqn:diagram-of-homs} and to show $G_{\CX}$ is an isomorphism. We do this in Corollary~\ref{cor:existence-of-G-R}
and 
Theorem~\ref{thm:G-R-is-an-isomorphism}, respectively.


\subsection{Mutation pairs and the subgroup \texorpdfstring{$N_{\CX}$}{NsubX}}
\label{sec:mutation-pairs}

We are in the situation of Setup~\ref{setup:all-indecomposable-objects-of-T-not-in-X-are-part-of-exchange-pair}. In this subsection we will define the subgroup $N_{\CX}$ and show the existence of the homomorphism $G_{\CX}$ appearing in \eqref{eqn:diagram-of-homs}.

\begin{defn}
\label{def:2d-calabi-yau}

Recall that $\field$ is an algebraically closed field. Let $\kdual(-) = \Hom_{\field}(-,\field)$ denote the standard $\field$-duality. 
Then $\CS$ is \emph{$2d$-Calabi-Yau} if, for all $S,S'\in\CS$, there are isomorphisms $\CS(S,S') \iso \kdual\CS(S',(\sus^{d})^{2}S)$, which are natural in each argument. 

\end{defn}

\begin{rem}
\label{rem:flX-is-abelian-and-inside-modX}

Note that under Setup~\ref{setup:all-indecomposable-objects-of-T-not-in-X-are-part-of-exchange-pair}, 
we have that $\CT$ is $n$-cluster tilting in $\CC$ and hence rigid (see Definition~\ref{def:n-cluster-tilting}), so the subcategory $\CX\sse\CT$ is also rigid. 
In addition, our situation meets all the conditions of \cite[Setup 1.1]{HolmJorgensen-generalized-friezes-and-a-modified-caldero-chapoton-map-depending-on-a-rigid-object-1}, and hence $\rMod{\CX}$ and $\rmod{\CX}$ share the same simple objects and finite length objects. In particular, $\fl{\CX}$ is abelian and the inclusions $\fl{\CX}\to \rmod{\CX}$ and $\fl{\CX}\to \rMod{\CX}$ are exact. 
Similar statements hold for $\CT$. 
See \cite[1.8]{HolmJorgensen-generalized-friezes-and-a-modified-caldero-chapoton-map-depending-on-a-rigid-object-1}.

\end{rem}

\begin{defn}
\label{def:exchange-pair}

\cite[(0.3)]{Jorgensen-tropical-friezes-and-the-index-in-higher-homological-algebra}
A pair $(S,S^{*})$ of objects in $\CS$ is said to be an \emph{exchange pair} if 
$\dim_{\field}\CS(S,\sus^{d}S^{*}) = 1$. 
\end{defn}

The following terminology we introduce is motivated by \cite[Subsec.\ 1.3]{OppermannThomas-higher-dimensional-cluster-combinatorics-and-representation-theory}.

\begin{defn}
\label{def:mutation-pair}

Let $T$ be an object of $\ind\CT$, and suppose $T^{*}$ lies in $\ind \CS$ with $T^{*} \niso T$. 
We call $(T,T^{*})$ a \emph{mutation pair} if 
it is an exchange pair (see Definition~\ref{def:exchange-pair}) 
and 
$\add((\ind\CT\setminus\{T\}) \cup \{T^{*}\})$ is again an Oppermann-Thomas cluster tilting subcategory of $\CS$ (see Definition~\ref{def:Oppermann-Thomas-cluster-tilting}). 
That is, replacing $T$ by $T^{*}$ yields another Oppermann-Thomas cluster tilting subcategory of $\CS$. 
In this case, we say 
$T^{*}$ is a \emph{mutation} of $T$. 

\end{defn}

For a mutation pair $(T,T^{*})$, there is a pair of distinguished $(d+2)$-angles that will play a key role in this whole section. A description of these can be given just as in \cite{OppermannThomas-higher-dimensional-cluster-combinatorics-and-representation-theory}.

\begin{prop}
\label{prop:mutation-is-unique-up-to-isomorphism}

Suppose $T\in\ind\CT$ has a mutation $T^{*}\in\ind\CS$. 
Then, up to (non-unique) isomorphism, 
there are \emph{exchange $(d+2)$-angles} 
\begin{align}
	\begin{tikzcd}[ampersand replacement=\&]
	Y^{0} \arrow{r}\& Y^{1} \arrow{r}\&\cdots\arrow{r}\& Y^{d}\arrow{r} \& Y^{d+1}\arrow{r}{\gamma^{T^{*}}} \& \sus^{d} Y^{0},
	\end{tikzcd} 			\label{eqn:exchange-d+2-angle-gamma-T-star}\\
	\begin{tikzcd}[ampersand replacement=\&]
	X^{0} \arrow{r} \& X^{1} \arrow{r}\&\cdots\arrow{r}\& X^{d}\arrow{r} \& X^{d+1}\arrow{r}{\gamma^{T}} \& \sus^{d} X^{0}
	\end{tikzcd}			\label{eqn:exchange-d+2-angle-gamma-T}
\end{align}
in $\CS$ that satisfy the following. 

\begin{enumerate}[\emph{(\roman*)}]

	\item $Y^{0} = X^{d+1} = T$ and $Y^{d+1} = X^{0} = T^{*}$.
	
	\item $\gamma^{T^{*}},\gamma^{T}$ are both non-zero. 
	
	\item $Y^{1}\to Y^{2}, Y^{2}\to Y^{3}, \ldots, Y^{d}\to Y^{d+1}$ 
	and $X^{0}\to X^{1}, X^{1}\to X^{2}, \ldots, X^{d-1}\to X^{d}$ are radical morphisms. 
	
	\item For all $1\leq i\leq d$, we have $X^{i},Y^{i}\in\add(\ind\CT\setminus\{T\})\sse\CT$.
	
	\item $Y^{0}\to Y^{1}$ is a minimal left $\add(\ind\CT\setminus\{T\})$-approximation.

	\item $X^{d}\to X^{d+1}$ is a minimal right $\add(\ind\CT\setminus\{T\})$-approximation.

\end{enumerate}
Moreover, $T^{*}$ is unique up to isomorphism. 

\end{prop}

\begin{proof}

As $\dim_{\field}\CS(T,\sus^{d}T^{*}) = 1$, we also have $\dim_{\field}\CS(T^{*},\sus^{d}T) = 1$ since $\CS$ is $2d$-Calabi-Yau. This gives the existence of 
\eqref{eqn:exchange-d+2-angle-gamma-T-star} and 
\eqref{eqn:exchange-d+2-angle-gamma-T} satisfying (i) and (ii). 
One can then follow the proof of $(1) \Rightarrow (2)\Rightarrow(3)$ of \cite[Thm.\ 5.7]{OppermannThomas-higher-dimensional-cluster-combinatorics-and-representation-theory} at our level of generality.
\end{proof}

\begin{rem}\label{rem:mutation}

\begin{enumerate}[label=(\roman*)]

	\item Suppose $T\in\ind\CT$ has a mutation $T^{*}$. 
Then $T^{*}$ is unique up to isomorphism by Proposition~\ref{prop:mutation-is-unique-up-to-isomorphism}. 
Thus, up to isomorphism, we speak of \emph{the} mutation $T^{*}$ of $T$. 
It also follows from the same result that $T^{*}\notin\CT$. 

	\item\label{item:not-mutable-object} As noted in Remark~\ref{rem:setups-in-sec-4}\ref{item:setup-4-1}, it is not guaranteed that every indecomposable object in $\CT$ has a mutation; see \cite[Rem.\ 5.8]{OppermannThomas-higher-dimensional-cluster-combinatorics-and-representation-theory}. 
	
\end{enumerate}
\end{rem}

The next definition is motivated by \cite[Def.\ 2.4]{HolmJorgensen-generalized-friezes-and-a-modified-caldero-chapoton-map-depending-on-a-rigid-object-2}, which deals with the triangulated version.

\begin{defn}
\label{def:N-X}

We define the subgroup 
\[ 
N_{\CX} = \Braket{
	\sum_{i=1}^{d} (-1)^{i} [X^{d+1-i}] - \sum_{i=1}^{d} (-1)^{i} [Y^{i}]  |   
	\begin{array}{l}
	\text{\eqref{eqn:exchange-d+2-angle-gamma-T-star} and \eqref{eqn:exchange-d+2-angle-gamma-T} are exchange } (d+2)\text{-angles} \\ 
	\text{with } T\in\ind\CT\setminus\ind\CX 
	\end{array}
	}
\] 
of $K_{0}^{\sp}(\CT)$, and let 
\[
Q_{N_{\CX}}\colon K_{0}^{\sp}(\CT)\to K_{0}^{\sp}(\CT)/N_{\CX}
\]
denote the canonical surjection.

\end{defn}

In Theorem~\ref{thm:d-index-gives-isomorphism-G-group-of-T-induced-exangulated-structure-on-S-to-split-G-group-of-T} we defined $L\colon K_{0}^{\sp}(\CT) \to K_{0}(\CS,\BE_{\CT},\fs_{\CT})$, which was shown to be the inverse of the $(d+2)$-angulated index 
$\dindx{\CT} \colon K_{0}(\CS,\BE_{\CT},\fs_{\CT}) \to K_{0}^{\sp}(\CT)$. 
We will show the subgroup $N_{\CX}$ vanishes under the composition $\wt{Q}L\colon K_{0}^{\sp}(\CT)\to K_{0}(\CS,\BE_{\CX},\fs_{\CX})$, in order to obtain a homomorphism 
$G_{\CX}\colon K_{0}^{\sp}(\CT)/N_{\CX} \to K_{0}(\CS,\BE_{\CX},\fs_{\CX})$. 
For this, we need the following key lemma.

\begin{lem}
\label{lem:gamma-T-gamma-T-star-vanish-under-F-wrt-T-and-X}

Suppose that $(T,T^{*})$ is a mutation pair with $T\in\ind\CT\setminus\ind\CX$. 
Let 
$ \gamma^{T^{*}}\colon T^{*} \to \sus^{d} T$ 
and 
$ \gamma^{T}\colon T \to \sus^{d} T^{*}$ 
be the morphisms in the exchange $(d+2)$-angles \eqref{eqn:exchange-d+2-angle-gamma-T-star} and \eqref{eqn:exchange-d+2-angle-gamma-T}. 
Then 
$F_{\CT}(\gamma^{T^{*}}) = 0 $ and $F_{\CX}(\gamma^{T^{*}}) = 0 = F_{\CX}(\gamma^{T})$. 

\end{lem}

\begin{proof}

As $\CT$ is Oppermann-Thomas cluster tilting and $\CX\sse\add(\ind\CT\setminus\{T\})\sse\CT$, we have that $\CS(\CT, \sus^{d}T) = 0 = \CS(\CX, \sus^{d}T)$. 
Thus, $F_{\CT}(\gamma^{T^{*}})$ and $F_{\CX}(\gamma^{T^{*}})$ both vanish. 
On the other hand, since replacing $T$ by $T^{*}$ in $\CT$ yields another Oppermann-Thomas cluster tilting subcategory of $\CS$ that also contains $\CX$, we have $\CS(\CX,\sus^{d}T^{*})= 0$. 
Hence, $F_{\CX}(\gamma^{T}) = 0$ as well.
\end{proof}

\begin{prop}
\label{prop:generators-of-N-R-are-trivial-in-G-group-of-R-induced-exangulated-structure-on-S}

Let $T\in\ind\CT \setminus \ind\CX$ and consider the exchange $(d+2)$-angles 
\eqref{eqn:exchange-d+2-angle-gamma-T-star}
and 
\eqref{eqn:exchange-d+2-angle-gamma-T}. 
Then each generator  
\[
\sum_{i=1}^{d} (-1)^{i} [X^{d+1-i}] - \sum_{i=1}^{d} (-1)^{i} [Y^{i}]
\]
of $N_{\CX}$ vanishes under $\wt{Q}L\colon K_{0}^{\sp}(\CT)\to K_{0}(\CS,\BE_{\CX},\fs_{\CX})$. 

\end{prop}

\begin{proof}

By Proposition~\ref{prop:mutation-is-unique-up-to-isomorphism}, we have that $X^{i},Y^{i}$ lie in $\CT$ for $1\leq i \leq d$. 
As $F_{\CX}(\gamma^{T^{*}}) = F_{\CX}(\gamma^{T}) = 0$ by Lemma~\ref{lem:gamma-T-gamma-T-star-vanish-under-F-wrt-T-and-X}, we see that 
\[
\begin{tikzcd}
Y^{0} \arrow{r}& Y^{1} \arrow{r}&\cdots\arrow{r}& Y^{d}\arrow{r} & Y^{d+1}\arrow[dashed]{r}{\gamma^{T^{*}}} & \sus^{d} Y^{0}, \\
X^{0} \arrow{r}& X^{1} \arrow{r}&\cdots\arrow{r}& X^{d}\arrow{r} & X^{d+1}\arrow[dashed]{r}{\gamma^{T}} & \sus^{d} X^{0}
\end{tikzcd}
\]
are $\fs_{\CX}$-distinguished $d$-exangles in the $d$-exangulated category $(\CS,\BE_{\CX},\fs_{\CX})$. 
So $\chi(Y^{\combul}) = \sum_{i=0}^{d+1} (-1)^{i} [Y^{i}]$ and 
$\chi(X^{\combul}) = \sum_{i=0}^{d+1} (-1)^{i} [X^{i}]$ are both trivial in $K_{0}(\CS,\BE_{\CX},\fs_{\CX})$. 
Then, since $Y^{0} = X^{d+1} = T$ and $Y^{d+1} = X^{0} = T^{*}$,  we see that
\begin{align*}
&\wt{Q}L\left(\sum_{i=1}^{d} (-1)^{i} [X^{d+1-i}] - \sum_{i=1}^{d} (-1)^{i} [Y^{i}]\right) \\
	&\hspace{1cm}=	\sum_{i=1}^{d} (-1)^{i} [X^{d+1-i}] - \sum_{i=1}^{d} (-1)^{i} [Y^{i}] \\
	&\hspace{1cm}= 	\sum_{i=0}^{d+1} (-1)^{i} [X^{d+1-i}] - (-1)^{0}[X^{d+1}] - (-1)^{d+1}[X^{0}] - \sum_{i=1}^{d} (-1)^{i} [Y^{i}] \\
	&\hspace{1cm}=	(-1)^{d+1}\chi(X^{\combul}) - (-1)^{0}[Y^{0}] - (-1)^{d+1}[Y^{d+1}] - \sum_{i=1}^{d} (-1)^{i} [Y^{i}] &&	\\
	&\hspace{1cm}= 	(-1)^{d+1}\chi(X^{\combul}) - \chi(Y^{\combul})\\
	&\hspace{1cm}= 	[0]. &&\qedhere
\end{align*}

\end{proof}

As an immediate consequence, we have:

\begin{cor}
\label{cor:existence-of-G-R}

There is a homomorphism $G_{\CX}\colon K_{0}^{\sp}(\CT)/N_{\CX} \to K_{0}(\CS,\BE_{\CX},\fs_{\CX})$ making the following diagram commute. 
\[
\begin{tikzcd}
K_{0}^{\sp}(\CT)
	\arrow[two heads]{r}{Q_{N_{\CX}}}
	\arrow{d}[swap]{L} 
&	K_{0}^{\sp}(\CT)/N_{\CX}
		\arrow[dotted]{d}{G_{\CX}}\\
K_{0}(\CS,\BE_{\CT},\fs_{\CT})
	\arrow[two heads]{r}{\wt{Q}}
&	K_{0}(\CS,\BE_{\CX},\fs_{\CX})
\end{tikzcd}
\]
\end{cor}

The commuting square in Corollary~\ref{cor:existence-of-G-R} provides the lower half of diagram \eqref{eqn:diagram-of-homs} in Theorem~\ref{thm:diagram-of-homomorphisms}. 
We prove in Subsection~\ref{sec:G-X-is-an-isomorphism} that $G_{\CX}$ is actually an isomorphism under the conditions of Setup~\ref{setup:F-lands-in-finite-length-modules}; see Theorem~\ref{thm:G-R-is-an-isomorphism}.


\subsection{A formula for \texorpdfstring{$\theta$}{theta} on simples}
\label{sec:formula-for-theta-on-simples}

In this subsection we derive a formula for the homomorphism $\theta\colon K_{0}(\rmod{\CT}) \to K_{0}^{\sp}(\CT)$ from Theorem~\ref{thm:Euler-sum-of-d-indices-equals-theta-of-image-of-connecting-map} on the simple objects in $\rmod{\CT}$. 
Recall that we are in the situation of Setup~\ref{setup:assume-Im-Fgamma-is-simple}. 
We remark that the condition of this setup is satisfied when $d=1$ and $\CT$ is part of a \emph{cluster structure} in the sense of \cite[p.\ 1039]{BuanIyamaReitenScott-cluster-structures-for-2-calabi-yau-categories-and-unipotent-groups}. 
In the context of Setup~\ref{setup:assume-Im-Fgamma-is-simple} in general, there is 
(see \cite[Prop.\ 2.3(b)]{Auslander-Rep-theory-of-Artin-algebras-II}) 
a bijective correspondence: 
\begin{align}
\Set{\begin{array}{c}\text{simple modules in } \rMod{\CT}\\ \text{ up to isomorphism}\end{array}} 	&\longleftrightarrow	\ind\CT\nonumber\\[1em]
\ol{S}_{T} = \CT(-,T)/ \rad_{\CT}(-,T)&\longleftrightarrow	T. \label{eqn:simple-at-T}
\end{align}
It follows from Remark~\ref{rem:flX-is-abelian-and-inside-modX} that, for $T\in\ind\CT\setminus\ind\CX$, we have 
$\ol{S}_{T} \in \rmod{\CT}$ and so $\theta([\ol{S}_{T}])$ makes sense. 
The next result is a generalisation of the triangulated version stated in \cite[1.5(ii)]{JorgensenPalu-a-caldero-chapoton-map-for-infinite-clusters}. It plays an important role in the proof of Theorem~\ref{thm:G-R-is-an-isomorphism}, the main result of Section~\ref{sec:K0TmodNX-as-a-grothendieck-group}.

\begin{prop}
\label{prop:theta-of-simple-in-modT-is-in-N-R}

Suppose $T\in\ind\CT\setminus\ind\CX$. 
Then we have
\[
\theta([\ol{S}_{T}]) = \sum_{i=1}^{d} (-1)^{i} [X^{d+1-i}] - \sum_{i=1}^{d} (-1)^{i} [Y^{i}],
\] 
where $Y^{\combul},X^{\combul}$ are the exchange $(d+2)$-angles \eqref{eqn:exchange-d+2-angle-gamma-T-star} and \eqref{eqn:exchange-d+2-angle-gamma-T}, respectively, for $T$. 
In particular, $\theta([\ol{S}_{T}])\in N_{\CX}$. 

\end{prop}

\begin{proof}

Let $T\in\ind\CT\setminus\ind\CX$, and denote by $T^{*}$ its mutation. 
Then $F_{\CT}(\gamma^{T^{*}}) = 0$ by Lemma~\ref{lem:gamma-T-gamma-T-star-vanish-under-F-wrt-T-and-X}. 
Thus, by Theorem~\ref{thm:Euler-sum-of-d-indices-equals-theta-of-image-of-connecting-map}, we have 
\begin{align*}
0 	&= \theta([\Im F_{\CT}(\gamma^{T^{*}})]) \\
	&= \sum_{i=0}^{d+1} (-1)^{i} \dindx{\CT}(Y^{d+1-i}) \\
	&= \sum_{i=1}^{d+1} (-1)^{i} [Y^{d+1-i}] + \dindx{\CT}(T^{*}).
\end{align*}
Rearranging and multiplying by $(-1)^{d+1}$ yields
\begin{align*}
(-1)^{d+1}\dindx{\CT}(T^{*}) 
	&=-\sum_{i=1}^{d+1} (-1)^{d+1-i} [Y^{d+1-i}] \\
	&= -\sum_{i=0}^{d} (-1)^{i} [Y^{i}] \\
	&= -\left( \sum_{i=1}^{d} (-1)^{i} [Y^{i}] + [T] \right).
\end{align*}
Then we see
\begin{align*}
\theta([\ol{S}_{T}]) 	
	&= \theta ([\Im F_{\CT}(\gamma^{T})]) && \text{by Setup~\ref{setup:assume-Im-Fgamma-is-simple}}\\
	&= \sum_{i=0}^{d+1} (-1)^{i} \dindx{\CT}(X^{d+1-i}) && \text{by Theorem~\ref{thm:Euler-sum-of-d-indices-equals-theta-of-image-of-connecting-map}} \\
	&= [T] + \sum_{i=1}^{d} (-1)^{i} [X^{d+1-i}] + (-1)^{d+1}\dindx{\CT}(T^{*}) && 
	\hspace{-5pt}
	\begin{array}[t]{l}
	\text{since } X^{d+1} = T, X^{0} = T^{*}\text{,}\\
	\text{and } X^{j}\in\CT \text{ for all } 1\leq j \leq d\\
	\text{by Proposition~\ref{prop:mutation-is-unique-up-to-isomorphism}}
	\end{array}\\
	&= [T] + \sum_{i=1}^{d} (-1)^{i} [X^{d+1-i}] - \left( \sum_{i=1}^{d} (-1)^{i} [Y^{i}] + [T] \right) && \text{from the observation above}\\
	&= \sum_{i=1}^{d} (-1)^{i} [X^{d+1-i}] - \sum_{i=1}^{d} (-1)^{i} [Y^{i}]. &&\qedhere
\end{align*}
\end{proof}


\subsection{\texorpdfstring{$G_{\CX}$}{G-X} is an isomorphism}
\label{sec:G-X-is-an-isomorphism}

In this subsection we are in the case of Setup~\ref{setup:F-lands-in-finite-length-modules}.
We will show that the homomorphism
\[
G_{\CX}\colon K_{0}^{\sp}(\CT)/N_{\CX} \to K_{0}(\CS,\BE_{\CX},\fs_{\CX}),
\]
which satisfies $\wt{Q} \circ L = G_{\CX} \circ Q_{N_{\CX}}$ (see Corollary~\ref{cor:existence-of-G-R}), is an isomorphism of abelian groups.

\begin{rem}
\label{remark:locally-bounded-implies-F-lands-in-finite-length-modules}
	Setup~\ref{setup:F-lands-in-finite-length-modules} implies that, for any morphism $\gamma\colon S^{d+1} \to \sus^{d}S^{0}$ in $\CS$, the (target of the) image $\Im F_{\CT}(\gamma)$ also lies in $\fl{\CT}$, because it is a subobject of the finite length object $F_{\CT}(\sus^{d}S^{0})$. 
\end{rem}

We can now prove our main result.

\begin{thm}
\label{thm:G-R-is-an-isomorphism}

$G_{\CX}
$ 
is an isomorphism.

\end{thm}

\begin{proof}

Firstly, since $\wt{Q}$ and $L$ are surjective (see \eqref{eqn:diagram-of-Qs} and Theorem~\ref{thm:d-index-gives-isomorphism-G-group-of-T-induced-exangulated-structure-on-S-to-split-G-group-of-T}, respectively), 
the equality $G_{\CX} Q_{N_{\CX}} = \wt{Q}L$ 
(see Corollary~\ref{cor:existence-of-G-R})
implies $G_{\CX}$ is surjective. 

To show $G_{\CX}$ is injective, let $g + N_{\CX} \in K_{0}^{\sp}(\CT)/N_{\CX}$ be an element such that $G_{\CX}(g + N_{\CX}) = 0$ in $K_{0}(\CS,\BE_{\CX},\fs_{\CX})$. We will show that $g + N_{\CX}$ is trivial in $K_{0}^{\sp}(\CT)/N_{\CX}$ , i.e.\ $g \in N_{\CX}$. 
As $G_{\CX}(g + N_{\CX}) = 0$ and $G_{\CX} Q_{N_{\CX}} = \wt{Q}L$, 
we have that 
\[
\wt{Q}(L(g))
	= G_{\CX} Q_{N_{\CX}} (g) 
	= G_{\CX}(g + N_{\CX}) 
	= 0.
\]
Since $\wt{Q}\colon K_{0}(\CS,\BE_{\CT},\fs_{\CT}) \to K_{0}(\CS,\BE_{\CX},\fs_{\CX}) = K_{0}^{\sp}(\CS)/\CI_{\CX}$ is the canonical surjection induced by the containment $\CI_{\CT}\sse \CI_{\CX}$, we see that $L(g) \in \CI_{\CX} + \CI_{\CT}$
in $K_{0}(\CS,\BE_{\CT},\fs_{\CT})$.

Recall that $\CI_{\CX}\leq K_{0}^{\sp}(\CS)$ (see Definition~\ref{def:subgroup-I-sub-D}) is generated by elements of the form 
$\chi(S^{\combul})$ (see Definition~\ref{def:euler-relation}), where $\lan S^{\combul} , \gamma\ran$ is an $\fs_{\CX}$-distinguished $d$-exangle
\[
\begin{tikzcd}
S^{0} \arrow{r}& S^{1} \arrow{r}&\cdots\arrow{r}& S^{d}\arrow{r} & S^{d+1}\arrow[dashed]{r}{\gamma} & \sus^{d} S^{0}
\end{tikzcd}
\]
in $(\CS,\BE_{\CX},\fs_{\CX})$. 
Thus, for some integer $s\geq 1$, we have 
\begin{equation}
\label{eqn:L-g-as-sum-of-chi-S-combuls}
L(g) 
	= \sum_{j=1}^{s} (-1)^{a_{j}}\chi(S_{j}^{\combul}) + \CI_{\CT}, 
\end{equation}
where, for each $1\leq j \leq s$, there is an $\fs_{\CX}$-distinguished $d$-exangle 
\[
\begin{tikzcd}
S_{j}^{0} \arrow{r}& S_{j}^{1} \arrow{r}&\cdots\arrow{r}& S_{j}^{d}\arrow{r} & S_{j}^{d+1}\arrow[dashed]{r}{\gamma_{j}} & \sus^{d} S_{j}^{0},
\end{tikzcd}
\]
and $a_{j}\in\{0,1\}$.
Hence, 
\begin{align*}
g &= \dindx{\CT}(L(g))&&\text{as } L^{-1} = \dindx{\CT} \text{ by Theorem~\ref{thm:d-index-gives-isomorphism-G-group-of-T-induced-exangulated-structure-on-S-to-split-G-group-of-T}}\\
	&=  \dindx{\CT} \Bigg(\sum_{j=1}^{s}(-1)^{a_{j}}\chi(S_{j}^{\combul}) + \CI_{\CT}\Bigg) &&\text{by \eqref{eqn:L-g-as-sum-of-chi-S-combuls}}\\
	&= \sum_{j=1}^{s} (-1)^{a_{j}} \sum_{i=0}^{d+1} (-1)^{i} \dindx{\CT}(S_{j}^{i}) &&\hspace{-5pt}\begin{array}[t]{l}
	\text{(note }\dindx{\CT}(\CI_{\CT})=0 \text{ since}\\K_{0}(\CS,\BE_{\CT},\fs_{\CT}) = K_{0}^{\sp}(\CS)/\CI_{\CT}\text{)}
	\end{array}\\
	&= (-1)^{d+1}\sum_{j=1}^{s}(-1)^{a_{j}} \sum_{i=0}^{d+1} (-1)^{i} \dindx{\CT}(S_{j}^{d+1-i})&& \text{(see Remark~\ref{rem:reversing-order-of-euler-relation})}\\
	&=(-1)^{d+1}\sum_{j=1}^{s}(-1)^{a_{j}} \theta([\Im F_{\CT}(\gamma_{j})]) && \text{by Theorem~\ref{thm:Euler-sum-of-d-indices-equals-theta-of-image-of-connecting-map}.}
\end{align*}

With this observation, it suffices to show that $\theta([\Im F_{\CT}(\gamma_{j})])\in N_{\CX}$ for each $1\leq j \leq s$. Thus, fix $j\in\{1,\ldots,s\}$. 
As $\Im F_{\CT}(\gamma_{j})$ is an object in $\fl\CT$, it admits a composition series 
\[
0= M_{0} < M_{1} < \cdots < M_{r-1} < M_{r} = \Im F_{\CT}(\gamma_{j})
\]
in $\rMod{\CT}$. 
Thus, for each $0\leq l \leq r-1$, there exists $T_{l}\in\ind\CT$ such that 
$M_{l+1}/M_{l}=\ol{S}_{T_{l}}$ is simple. 

We claim that $T_{l}\notin \ind\CX$ for all $0\leq l \leq r-1$. 
Assume, for contradiction, that $T_{l}\in\ind \CX$ for some $l\in\{0,\ldots, r-1\}$. 
Denote the simple objects in $\rmod{\CX}$ by $S_{X}$ for an object $X\in\ind\CX$. 
As shown in \cite[2.5]{HolmJorgensen-generalized-friezes-and-a-modified-caldero-chapoton-map-depending-on-a-rigid-object-2}, the inclusion $\iota\colon \CX \to \CT$ induces an exact functor $\iota^{*}\colon \fl\CT \to \fl\CX$, which, on simple objects, is given by 
\[
\iota^{*} \ol{S}_{T} = 	\begin{cases}
				S_T	& \text{if } T\in\ind\CX\\
				0	& \text{else.}
				\end{cases}
\]
In particular, we have $\iota^{*}M_{l+1}$ is a submodule of 
$\iota^{*}\Im F_{\CT}(\gamma_{j}) 
	= \Im F_{\CX}(\gamma_{j})$. 
But $F_{\CX}(\gamma_{j})$ is 0 since $\gamma_{j}$ factors through $\CX^{\perp_{0}}$, so 
$S_{T_{l}} 
= \iota^{*}\ol{S}_{T_{l}}
= \iota^{*}(M_{l+1}/M_{l})
= 0$, 
which is a contradiction as $S_{T_{l}}$ is simple. 
This proves the claim.

Using the composition series for $\Im F_{\CT}(\gamma_{j})$, we have 
\[
\theta([\Im F_{\CT}(\gamma_{j})]) 
= \theta([\ol{S}_{T_{0}}]) + \cdots + \theta([\ol{S}_{T_{r-1}}]).
\]
For each $l\in\{0,\ldots,r-1\}$ we have $\theta([\ol{S}_{T_{l}}])\in N_{\CX}$ by Proposition~\ref{prop:theta-of-simple-in-modT-is-in-N-R}, and hence $\theta([\Im F_{\CT}(\gamma_{j})])$ also lies in $N_{\CX}$. 
Since $j\in\{1,\ldots,s\}$ was arbitrary, we see that $g\in N_{\CX}$.
Hence,  $G_{\CX}$ is injective and, moreover, an isomorphism.
\end{proof}


\section{The \texorpdfstring{$\CX$}{X}-Caldero-Chapoton map}
\label{sec:re-modified-CC-map}

In this section, we describe the connection between our results and the modified Caldero-Chapoton map defined by Holm and J\o{}rgensen in \cite{HolmJorgensen-generalized-friezes-and-a-modified-caldero-chapoton-map-depending-on-a-rigid-object-2}.

\begin{setup}
\label{setup:5}

In this section, we suppose the following.

\begin{enumerate}[(\roman*)]

	\item $\field$ is an algebraically closed field.
	
	\item $\CC$ is a skeletally small, $\field$-linear, Krull-Schmidt, $\Hom$-finite, 2-Calabi-Yau triangulated category.
	
	\item 
	$\CX$ is a functorially finite, additive subcategory of $\CC$. 
	
	\item $\CT$ is a locally bounded, cluster tilting (i.e.\ 2-cluster tilting) subcategory of $\CC$, which belongs to a cluster structure in the sense of \cite[p.\ 1039]{BuanIyamaReitenScott-cluster-structures-for-2-calabi-yau-categories-and-unipotent-groups}.

	\item $\CX\sse \CT \sse \CC$. 

\end{enumerate}
\end{setup}

\begin{rem}
\label{rem:setup-5-implies-setups-3-and-4s}

Note that Setup~\ref{setup:5} implies: Setup~\ref{setup:initial-d-setup-C-T-S} is met with $d = 1$, $n = 2$ and $\CS = \CC$;  Setups~\ref{setup:all-indecomposable-objects-of-T-not-in-X-are-part-of-exchange-pair} and \ref{setup:assume-Im-Fgamma-is-simple} hold as $\CT$ is part of a cluster structure; and Setup~\ref{setup:F-lands-in-finite-length-modules} holds as $\CT$ is assumed to be locally bounded. 
Thus, all results from Sections~\ref{sec:cluster-tilting-subcategories} and \ref{sec:K0TmodNX-as-a-grothendieck-group} apply. 
Furthermore, it also follows from $\CT$ being part of a cluster structure that all indecomposables in $\CT$ have a mutation. 

\end{rem}

Over \cite{HolmJorgensen-generalized-friezes-and-a-modified-caldero-chapoton-map-depending-on-a-rigid-object-1} and \cite{HolmJorgensen-generalized-friezes-and-a-modified-caldero-chapoton-map-depending-on-a-rigid-object-2}, Holm and J\o{}rgensen constructed a modified Caldero-Chapoton map, which generalised the original of \cite{CalderoChapoton-cluster-algebras-as-hall-algebras-of-quiver-representations} in two ways: it only depends on a rigid subcategory and can take values in a general commutative ring $A$. This modified Caldero-Chapoton map depends on several other functions, which we recall now following \cite[Sec.\ 2]{HolmJorgensen-generalized-friezes-and-a-modified-caldero-chapoton-map-depending-on-a-rigid-object-2}. 

In our setup, $d=1$ so there are exchange triangles
\begin{align}
	\begin{tikzcd}[ampersand replacement=\&]
	T \arrow{r}\& Y \arrow{r} \& T^{*}\arrow{r}{\gamma^{T^{*}}} \& \sus T,
	\end{tikzcd} 			\label{eqn:exchange-triangle-gamma-T-star}\\
	\begin{tikzcd}[ampersand replacement=\&]
	T^{*}\arrow{r} \& X \arrow{r} \& T\arrow{r}{\gamma^{T}} \& \sus T^{*}
	\end{tikzcd}			\label{eqn:exchange-triangle-gamma-T}
\end{align}
for each $T\in\ind\CT$; see Subsection~\ref{sec:mutation-pairs}. 
Then the subgroup $N_{\CX} \leq K_{0}^{\sp}(\CT)$ from Definition~\ref{def:N-X} can now be written more simply as
\begin{equation}
\label{eqn:definition-of-N-X-in-Section-5}
N_{\CX} = \Braket{
	[X] - [Y] |   
	\begin{array}{l}
	\text{\eqref{eqn:exchange-triangle-gamma-T-star} and \eqref{eqn:exchange-triangle-gamma-T} are exchange triangles} \\ 
	\text{with } T\in\ind\CT\setminus\ind\CX 
	\end{array}
	},
\end{equation}
and we still have the quotient homomorphism 
$Q_{N_{\CX}} \colon K_{0}^{\sp}(\CT) \to  K_{0}^{\sp}(\CT) / N_{\CX}$.

Let $\iota\colon \CX \to \CT$ be the inclusion functor and 
$\iota^{*}\colon \fl\CT \to \fl\CX$ the induced exact functor, 
as in the proof of Theorem~\ref{thm:G-R-is-an-isomorphism}. 
Then $\iota^{*}$ induces a surjective homomorphism 
\[
\kappa\colon K_{0}(\fl \CT)\onto K_{0}(\fl \CX)
\]
on the level of Grothendieck groups, given by 
\[
\kappa([\ol{S}_{T}]) = 	\begin{cases}
					[S_T]	& \text{if } T\in\ind\CX\\
					0	& \text{if } T\in\ind\CT\setminus\ind\CX\text{.}
					\end{cases}
\]
Furthermore, there is a homomorphism 
\[
\ol{\phi}\colon K_{0}(\fl\CT) \to K_{0}^{\sp}(\CT)
\]
defined by 
\begin{equation}
\label{eqn:definition-of-phi-bar}
\ol{\phi}([\ol{S}_{T}]) = [Y] - [X]
\end{equation}
for each $T\in\ind\CT$, where $X,Y$ are as in \eqref{eqn:exchange-triangle-gamma-T-star} and \eqref{eqn:exchange-triangle-gamma-T}; 
see \cite[1.5(ii)]{JorgensenPalu-a-caldero-chapoton-map-for-infinite-clusters}. 
Post-composing with $L\colon K_{0}^{\sp}(\CT) \to K_{0}(\CC,\BE_{\CT},\fs_{\CT})$ 
(see e.g.\ 
Theorem~\ref{thm:diagram-of-homomorphisms}) gives a homomorphism 
\[
\ol{\psi}\deff L\ol{\phi}\colon  K_{0}(\fl\CT)\to K_{0}(\CC,\BE_{\CT},\fs_{\CT}),
\]
such that 
\[
\ol{\psi}([\ol{S}_{T}]) 
	= L([Y]) - L([X]) 
	= L\dindx{\CT}(Y) - L\dindx{\CT}(X) 
	= Q_{\CT}([Y]) - Q_{\CT}([X]).
\] 
Moreover, note that, for any $T\in\ind\CT\setminus\ind\CX$, we have 
\begin{align*}
\wt{Q}\ol{\psi}([\ol{S}_{T}]) 
	&= \wt{Q}L\ol{\phi}([\ol{S}_{T}]) && \text{as } \ol{\psi} = L\ol{\phi}\\
	&= G_{\CX} Q_{N_{\CX}} \ol{\phi}([\ol{S}_{T}]) && \text{by Theorem~\ref{thm:diagram-of-homomorphisms}}\\
	&= G_{\CX} Q_{N_{\CX}} ([Y] - [X]) && \text{by \eqref{eqn:definition-of-phi-bar}} \\
	&= 0 && \text{using \eqref{eqn:definition-of-N-X-in-Section-5}.}
\end{align*}
Therefore, $\wt{Q}\ol{\psi}$ factors uniquely through $\kappa$, i.e.\ there is a unique homomorphism 
\[
\psi \colon K_{0}(\fl \CX) \to K_{0}(\CC,\BE_{\CX},\fs_{\CX})
\]
making the following diagram commute. 
\begin{equation}\label{eqn:property-of-psi}
\begin{tikzcd}[row sep=0.3cm, column sep=1.2cm]
&   K_{0}(\fl\CT) \arrow{dl}[swap]{\ol{\phi}}\arrow{dd}{\ol{\psi}} \arrow[two heads]{r}{\kappa}&   K_{0}(\fl\CX)\arrow[dotted]{dd}{\psi}\\
K_{0}^{\sp}(\CT)\arrow{dr}[swap]{L} & &\\
&K_{0}(\CC,\BE_{\CT},\fs_{\CT}) \arrow[two heads]{r}{\wt{Q}}& K_{0}(\CC,\BE_{\CX},\fs_{\CX})
\end{tikzcd}
\end{equation}

\begin{defn}
\label{def:X-CC-map}

(cf.\ \cite[1.2]{HolmJorgensen-generalized-friezes-and-a-modified-caldero-chapoton-map-depending-on-a-rigid-object-2}) 
Let $A$ be a commutative ring. 
Suppose 
\[
\ol{\eps} \colon K_{0}(\CC,\BE_{\CX},\fs_{\CX}) \to A
\]
is a map, such that 
\begin{equation}\label{eqn:exponential-properties-of-eps}
\ol{\eps}(0) = 1
\hspace{1cm}\text{and} \hspace{1cm} 
\ol{\eps} (e+f) = \ol{\eps}(e)\ol{\eps}(f).
\end{equation} 
Define $\alpha \colon \obj(\CC)\to A$ and $\beta \colon K_{0}(\fl \CX)\to A$ as follows: 
\[
\alpha \deff \ol{\eps} \circ Q_{\CX}
\hspace{1cm}\text{and} \hspace{1cm}
\beta \deff \ol{\eps}  \circ \psi\text{.} 
\]
Let $C\in\obj(\CC)$. Then the \emph{$\CX$-Caldero-Chapoton map} is given by the formula
\[
\rho(C) 
= \alpha(C) \sum_{e}\chi\left(\text{Gr}_{e}(F_{\CX}(\sus C))\right) \beta(e),
\]
where $\text{Gr}_{e}(F_{\CX}(\sus C)) = \text{Gr}_{e}(\restr{\CC(-,\sus C)}{\CX})$ is the Grassmannian of finite length submodules $M\sse \restr{\CC(-,\sus C)}{\CX}$ in $\rMod{\CX}$ with $[M] = e\in K_{0}(\fl \CX)$, and $\chi$ is the Euler characteristic defined by \'{e}tale cohomology with proper support.

\end{defn}

\begin{rem}
\label{rem:5-4}

\begin{enumerate}[label=(\roman*)]

	\item 
	\label{remark-5-4-i}
	It is not immediately apparent that the formula for $\rho$ above is well-defined for all $C\in\obj(\CC)$. 
	However, if $F_{\CX}(\sus C)$ has finite length, then the sum over $e$ is finite and $\rho(C)$ gives an element of $A$.
	Recall that Setup~\ref{setup:F-lands-in-finite-length-modules} holds since we are in the situation of Setup~\ref{setup:5} (see Remark~\ref{rem:setup-5-implies-setups-3-and-4s}). Thus, $F_{\CT}(\sus C)$ has finite length for all $C\in\obj(\CC)$, and so 
	$\kappa (F_{\CT}(\sus C)) = F_{\CX}(\sus C)$ indeed has finite length. 

	\item The definition of $\rho$ is precisely how the original Caldero-Chapoton map is given in \cite{CalderoChapoton-cluster-algebras-as-hall-algebras-of-quiver-representations}, but with the involved maps and functors adjusted relative to the functorially finite, rigid subcategory $\CX\sse\CT$. 
	Indeed, replacing 
	$Q_{\CX}$, 
	$\psi$ 
	and 
	$\ol{\eps}$, respectively, 
	by $\indx{\CT}$, 
	$\ol{\phi}$, 
	and 
	a map $\ol{\eps}\colon K_{0}^{\sp}(\CT)\to A$ (still satisfying condition \eqref{eqn:exponential-properties-of-eps}), where $A$ is the Laurent polynomial ring $\BZ[\tensor[]{x}{_{T}}, \tensor[]{x}{_{T}^{-1}}\tensor[]{]}{_{T\in\ind{\CT}}}$, recovers the version of the Caldero-Chapoton map given in \cite[1.8]{JorgensenPalu-a-caldero-chapoton-map-for-infinite-clusters}. 

\end{enumerate}

\end{rem}

\begin{thm}

The $\CX$-Caldero-Chapoton map defined above recovers the modified Caldero-Chapoton map of Holm--J\o{}rgensen \cite{HolmJorgensen-generalized-friezes-and-a-modified-caldero-chapoton-map-depending-on-a-rigid-object-2}. 

\end{thm}

\begin{proof}

To show this, let $\eps\colon K_{0}^{\sp}(\CT)/N_{\CX} \to A$ be a map satisfying $\eps(0) = 1$ and $\eps (e+f) = \eps(e)\eps(f)$, as assumed in \cite[Def.\ 2.8]{HolmJorgensen-generalized-friezes-and-a-modified-caldero-chapoton-map-depending-on-a-rigid-object-2}.  
Define $\ol{\eps} \deff \eps\tensor[]{G}{_{\CX}^{-1}}$. Then $\ol{\eps}$ satisfies the requirements for the definition of the $\CX$-Caldero-Chapoton map. 

The map $\alpha$ is 
$\alpha 
	= \ol{\eps} \circ Q_{\CX} 
	= \eps\tensor[]{G}{_{\CX}^{-1}}Q_{\CX}
	= \eps Q_{N_{\CX}} \dindx{\CT}$, 
which is readily seen to agree with \cite[Def.\ 2.8]{HolmJorgensen-generalized-friezes-and-a-modified-caldero-chapoton-map-depending-on-a-rigid-object-2}. 
For $\beta 
	= \ol{\eps}  \circ \psi 
	= \eps\tensor[]{G}{_{\CX}^{-1}}\psi$, 
first note that in \cite[Def.\ 2.8]{HolmJorgensen-generalized-friezes-and-a-modified-caldero-chapoton-map-depending-on-a-rigid-object-2} the map is defined as $\eps\phi$, where $\phi\colon K_{0}(\fl{\CX}) \to K_{0}^{\sp}(\CT)/N_{\CX}$ is the unique homomorphism satisfying $\phi \kappa = Q_{N_{\CX}}\ol{\phi}$. Thus, it suffices to show that $\tensor[]{G}{_{\CX}^{-1}}\psi = \phi$. 
This does indeed hold:
\begin{align*}
\tensor[]{G}{_{\CX}^{-1}}\psi\kappa 
	&= \tensor[]{G}{_{\CX}^{-1}}\wt{Q}\ol{\psi} && \text{using } \eqref{eqn:property-of-psi}\\
	&= \tensor[]{G}{_{\CX}^{-1}}\wt{Q}L\ol{\phi} && \text{since } \ol{\psi} = L\ol{\phi}\\
	&= Q_{N_{\CX}}\ol{\phi} && \text{using Theorem~\ref{thm:diagram-of-homomorphisms},}
\end{align*}
and hence the claim follows by uniqueness.
\end{proof}

\begin{rem}
\label{rem:psi-is-theta-X-mu}

We conclude this section with a remark about how our considerations have a connection with recent work by the authors on indices with respect to rigid subcategories. 
The map 
$\psi \colon K_{0}(\fl \CX) \to K_{0}(\CC,\BE_{\CX},\fs_{\CX})$ 
obtained above in \eqref{eqn:property-of-psi} relates to the homomorphism 
$\theta_{\CX}\colon K_{0}(\rmod{\CX}) \to K_{0}(\CC,\BE_{\CX},\fs_{\CX})$ 
defined recently in \cite{JorgensenShah-the-index-with-respect-to-a-rigid-subcategory}. 
The map $\theta_{\CX}$ is shown to measure how far the index with respect to $\CX$ is from being additive on triangles. 

The inclusion $\fl{\CX}\to\rmod{\CX}$ (see Remark~\ref{rem:flX-is-abelian-and-inside-modX}) induces a homomorphism 
\[
\mu \colon K_{0}(\fl{\CX}) \to K_{0}(\rmod{\CX}).
\]
Recall that $F_{\CT}(C)\in\fl{\CT}$ for all $C\in\CC$ (see Remark~\ref{rem:5-4}\ref{remark-5-4-i}). 
Consider the $K_{0}$-classes $[F_{\CA}(C)]_{\fl{\CA}}$ and $[F_{\CA}(C)]_{\rmod{\CA}}$ in 
$K_{0}(\fl{\CA})$ and $K_{0}(\rmod{\CA})$, respectively, for 
$\CA\in\{\CX, \CT\}$. 
Then we have 
\begin{align}
\psi ([F_{\CX}(C)]_{\fl{\CX}})
	&= \psi \kappa ([F_{\CT}(C)]_{\fl{\CT}}) \nonumber\\
	&= \wt{Q}\ol{\psi} ([F_{\CT}(C)]_{\fl{\CT}})\nonumber&& \text{using \eqref{eqn:property-of-psi}}\\
	&= \wt{Q}L(\indx{\CT}(C) + \indx{\CT}(\sus^{-1}C)) && \text{by  \cite[Lem.\ 2.10]{HolmJorgensen-generalized-friezes-and-a-modified-caldero-chapoton-map-depending-on-a-rigid-object-2}}\nonumber\\
	&= Q_{\CX}([C]) + Q_{\CX}([\sus^{-1} C]) && \text{using Theorem~\ref{thm:diagram-of-homomorphisms},} \label{eqn:psi-is-restriction-of-Q-sub-X}
\end{align}
viewing $[C]$ and $[\sus^{-1} C]$ as elements of $K_{0}^{\sp}(\CS)$. 
Note that the term in \eqref{eqn:psi-is-restriction-of-Q-sub-X} is equal to $\theta_{\CX}([F_{\CX}(C)]_{\rmod{\CX}})$, where $\theta_{\CX}\colon K_{0}(\rmod{\CX}) \to K_{0}(\CC,\BE_{\CX},\fs_{\CX})$ is the homomorphism of \cite{JorgensenShah-the-index-with-respect-to-a-rigid-subcategory}. 
Therefore, we see that $\psi ([F_{\CX}(C)]_{\fl{\CX}}) = \theta_{\CX}\mu([F_{\CX}(C)]_{\fl{\CX}})$.
In particular, this suggests $\psi$ can also be seen as some measure of how far the index with respect to $\CX$ defined in \cite{JorgensenShah-the-index-with-respect-to-a-rigid-subcategory} is from being additive on triangles. 

\end{rem}

\begin{acknowledgements}
The authors would like to thank P.-G.\ Plamondon for informing the first author of \cite[Prop.\ 4.11]{PadrolPaluPilaudPlamondon-associahedra-for-finite-type-cluster-algebras-and-minimal-relations-between-g-vectors}. This motivated the present article. 
The authors also thank M.\ Gorsky, H.\ Nakaoka and Y.\ Palu for directing the authors to \cite[App.\ B]{Fiorot-n-quasi-abelian-categories-vs-n-tilting-torsion-pairs}. 
We are grateful to M.\ Pressland for finding an error in a previous version of the article, and to the referee for their comments.

This work was supported by a 
DNRF Chair from the Danish National Research Foundation (grant DNRF156), 
by a Research Project 2 from the Independent Research Fund Denmark (grant 1026-00050B), 
by the Aarhus University Research Foundation (grant AUFF-F-2020-7-16), 
and by the Engineering and Physical Sciences Research Council (grant EP/P016014/1). 
The second author is also grateful to the London Mathematical Society for funding through an Early Career Fellowship with support from Heilbronn Institute for Mathematical Research (grant ECF-1920-57). 
\end{acknowledgements}

{
\bibliography{references2}
\bibliographystyle{mybst}}
\end{document}